\documentclass[12pt]{amsart}
\newtheorem{thm}{Theorem}
\newtheorem{cor}{Corollary}
\newtheorem{lemma}{Lemma}
\newtheorem{prop}{Proposition}
\newcommand{\Rho}{\mathrm{P}}
\DeclareSymbolFont{script}{U}{eus}{m}{n}
\DeclareMathSymbol{\Wedge}{0}{script}{"5E}
\newcommand{\oooodots}[4]{\begin{picture}(72,11)
\put(4,1.5){\makebox(0,0){$\bullet$}}
\put(20,1.5){\makebox(0,0){$\bullet$}}
\put(36,1.5){\makebox(0,0){$\bullet$}}
\put(52,1.5){\makebox(0,0){$\bullet$}}
\put(66,1){\makebox(0,0){$\cdots$}}
\put(4,1.5){\line(1,0){53}}
\put(4,8){\makebox(0,0){$\scriptstyle #1$}}
\put(20,8){\makebox(0,0){$\scriptstyle #2$}}
\put(36,8){\makebox(0,0){$\scriptstyle #3$}}
\put(52,8){\makebox(0,0){$\scriptstyle #4$}}
\end{picture}}
\newcommand{\xooodots}[4]{\begin{picture}(72,11)
\put(4,1.5){\makebox(0,0){$\times$}}
\put(20,1.5){\makebox(0,0){$\bullet$}}
\put(36,1.5){\makebox(0,0){$\bullet$}}
\put(52,1.5){\makebox(0,0){$\bullet$}}
\put(66,1){\makebox(0,0){$\cdots$}}
\put(4,1.5){\line(1,0){53}}
\put(4,8){\makebox(0,0){$\scriptstyle #1$}}
\put(20,8){\makebox(0,0){$\scriptstyle #2$}}
\put(36,8){\makebox(0,0){$\scriptstyle #3$}}
\put(52,8){\makebox(0,0){$\scriptstyle #4$}}
\end{picture}}
\newcommand{\oooo}[4]{\begin{picture}(56,11)
\put(4,1.5){\makebox(0,0){$\bullet$}}
\put(20,1.5){\makebox(0,0){$\bullet$}}
\put(36,1.5){\makebox(0,0){$\bullet$}}
\put(52,1.5){\makebox(0,0){$\bullet$}}
\put(4,1.5){\line(1,0){48}}
\put(4,8){\makebox(0,0){$\scriptstyle #1$}}
\put(20,8){\makebox(0,0){$\scriptstyle #2$}}
\put(36,8){\makebox(0,0){$\scriptstyle #3$}}
\put(52,8){\makebox(0,0){$\scriptstyle #4$}}
\end{picture}}
\newcommand{\xoox}[4]{\begin{picture}(56,11)
\put(4,1.5){\makebox(0,0){$\times$}}
\put(20,1.5){\makebox(0,0){$\bullet$}}
\put(36,1.5){\makebox(0,0){$\bullet$}}
\put(52,1.5){\makebox(0,0){$\times$}}
\put(4,1.5){\line(1,0){48}}
\put(4,8){\makebox(0,0){$\scriptstyle #1$}}
\put(20,8){\makebox(0,0){$\scriptstyle #2$}}
\put(36,8){\makebox(0,0){$\scriptstyle #3$}}
\put(52,8){\makebox(0,0){$\scriptstyle #4$}}
\end{picture}}
\newsavebox{\innn}
\sbox{\innn}{\begin{picture}(0,0)
\put(0,6){\oval(6,12)[b]}
\put(0,0){\line(0,1){6}}
\end{picture}}
\def\inn{\usebox{\innn}}
\begin{document}
\title[Special metrics]
{Special metrics and scales in parabolic geometry}
\author[M.G. Eastwood]{Michael Eastwood}
\address{\hskip-\parindent
School of Mathematical Sciences\\
University of Adelaide\\ 
SA 5005\\ 
Australia}
\email{meastwoo@gmail.com}
\author[L. Zalabov\'a]{Lenka Zalabov\'a}
\address{\hskip-\parindent
Department of Mathematics and Statistics\\
Faculty of Science\\ 
\newline Masaryk University\\
Kotl\'a\v{r}sk\'a 2\\ 
Brno 611 37\\
Czech Republic\\
\newline\mbox{\rm and}
\newline Institute of Mathematics\\ Faculty of Science\\
University of South Bohemia\\
Brani\v{s}ovsk\'a 1760\\
\v{C}esk\'e Bud\v{e}jovice 370 05\\
Czech Republic}
\email{lzalabova@gmail.com}
\subjclass{53A20,53A30,53A40}
\begin{abstract} Given a parabolic geometry, it is sometimes possible to find
special metrics characterised by some invariant conditions. In conformal
geometry, for example, one asks for an Einstein metric in the conformal class.
Einstein metrics have the special property that their geodesics are  
{\em distinguished\/}, as unparameterised curves, in the sense of parabolic
geometry. This property characterises the Einstein metrics. In this article we 
initiate a study of corresponding phenomena for other parabolic geometries, in 
particular for the hypersurface CR and contact Legendrean cases.   
\end{abstract}
\renewcommand{\subjclassname}{\textup{2010} Mathematics Subject Classification}

\thanks{This work was supported by the Simons Foundation grant 
346300 and the Polish Government MNiSW 2015--2019 matching fund. It was 
started in 2017 whilst the authors were visiting the Banach Centre at IMPAN in 
Warsaw for the Simons Semester `Symmetry and Geometric Structures.'}

\thanks{During 2017--2019 the second author was also supported by grant
no.~17-01171S from the Czech Science Foundation entitled `Invariant
differential operators and their applications in geometric modelling and
control theory.' During 2020--2022, she is supported by grant no.~20-11473S,
`Symmetry and invariance in analysis, geometric modelling and control theory,'
also from the Czech Science Foundation.}

\thanks{We should also like to thank Green Caff\`e Nero, Pi\c{e}kna 18, Warsaw
for providing a very good working environment and quick access to excellent
coffee.}

\maketitle

\setcounter{section}{-1}
\section{Introduction}\label{introduction}
Let $M$ be a smooth $n$-manifold. A {\em projective structure\/} on $M$ is an
equivalence class of torsion-free connections on the tangent bundle~$TM$, where
two connections are said to be equivalent if and only if locally they have the
same geodesics as unparameterised curves. The study of such structures is
called {\em projective differential geometry\/}. It is a simple example of a
general class of structures called {\em parabolic differential
geometries\/}~\cite{parabook}. One can ask whether there is a Riemannian metric
on $M$ such that the geodesics of this metric are the geodesics of the
projective structure. It is a classical question known as the 
{\em metrisability problem\/} for projective structures. It was already known
in 1889 that there are obstructions in two dimensions:
R.~Liouville~\cite{Liouville} showed that the problem is governed by an
overdetermined linear system of partial differential equations. The
two-dimensional case is completely solved in~\cite{BDE}. In all dimensions, the
metrisability problem is governed by an overdetermined linear
operator~\cite{EM}. It is a first operator from the projective
Bernstein-Gelfand-Gelfand sequence~\cite{CD,CSS}, specifically
$$\sigma^{bc}\longmapsto
\mbox{the trace-free part of }\nabla_a\sigma^{bc},$$
where $\sigma^{bc}$ is a symmetric tensor of projective weight~$-2$.
Metrisability is obstructed in all dimensions~\cite{BDE} and explicit
obstructions in three dimensions are given in~\cite{DE}.

A {\em conformal structure\/} on $M$ is an equivalence class of Riemannian
metrics on~$M$, where two such metrics $g_{ab}$ and $\widehat{g}_{ab}$ are
equivalent if and only if $\widehat{g}_{ab}=\Omega^2g_{ab}$ for some positive
smooth function~$\Omega$. For $n\geq3$, as we shall henceforth suppose, it is
another example of a parabolic differential geometry. Of course, in this case
there is no question that there are metrics underlying the structure. But there
are special curves in conformal geometry that are invariantly defined by the
structure, and which play the r\^{o}le of the unparameterised geodesics in
projective geometry. These are the {\em conformal circles\/}~\cite{BE,T,Y},
here to be regarded as unparameterised curves. Unlike the geodesics of a
projective structure, they are defined by a higher order jet at any one point:
in a way that will soon be made precise, one needs to know both an initial
direction and acceleration to determine a conformal circle. Nevertheless, we
can ask the question of whether there is a metric in the conformal class so
that all its geodesics are conformal circles. We shall find that this question
is equivalent to the classical question of whether the conformal structure
admits an Einstein representative. So this is a non-trivial question only for
$n\geq4$ with complex projective space ${\mathbb{CP}}_2$, equipped with its
Fubini-Study metric, providing a good example where all geodesics are conformal
circles. The existence of an Einstein metric in a given conformal class is also
governed by a first BGG operator, specifically
$$\sigma\longmapsto
\mbox{the trace-free part of }(\nabla_a\nabla_b\sigma+\Rho_{ab}\sigma),$$
where $\sigma$ has conformal weight $1$ and $\Rho_{ab}$ is the Schouten
tensor~(\ref{schouten}). The na\"{\i}ve restriction on a positive smooth
function $\Omega$ such that $\Omega^2g_{ab}$ is Einstein is nonlinear and was
given by Brinkmann~\cite[Eq.~(2.26)]{B} but LeBrun~\cite[p.~558]{LeB} observes
that the equation on $\sigma\equiv\Omega^{-1}$ is linear.

The conformal circles on the round $n$-sphere are the genuine round circles 
for the unit sphere $S^n\subset{\mathbb{R}}^{n+1}$ sliced by any plane 
in~${\mathbb{R}}^{n+1}$, not necessarily through the origin. A more congenial 
viewpoint is to regard the round $n$-sphere inside ${\mathbb{RP}}_{n+1}$ 
according to
\begin{equation}\label{flat_model_conformal}
S^n=\{[x^0,x^1,\cdots,x^n,x^\infty]\mbox{ s.t.\ }
2x^0x^\infty=(x^1)^2+(x^2)^2\cdots+(x^n)^2\}.\end{equation}
Its main feature is that the evident action of ${\mathrm{SO}}(n+1,1)$ is by
conformal transformations and that all conformal transformations of $S^n$ arise
in this way (see, for example,~\cite{EG} for details). In this model, the
conformal circles are the intersections of $S^n$ with projective
planes~${\mathbb{RP}}_2$, linearly embedded in~${\mathbb{RP}}_{n+1}$.
$$\setlength{\unitlength}{1.5pt}
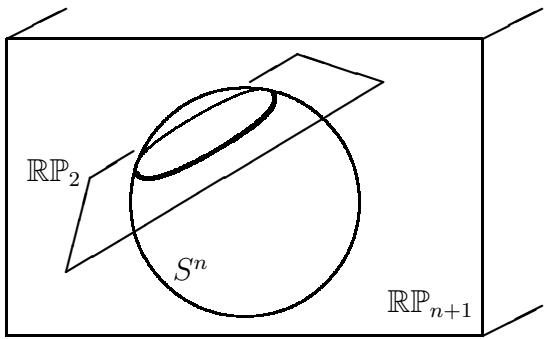
\begin{picture}(40,90)(0,-10)
\linethickness{.4pt}
\qbezier (37.32,55) (34.82,59.33) (17.50,49.33)
\qbezier (17.50,49.33) (0.18,39.33) (2.68,35)
\linethickness{1.3pt}
\qbezier (2.68,35) (5.18,30.67) (22.5,40.67)
\qbezier (22.5,40.67) (39.82,50.67) (37.32,55)
\put(12.3,-11.7){\setlength{\unitlength}{.8cm}
\begin{picture}(5,4.6)(-1,-2.6)
\put(-1.2,-1.3){$S^n$}
\linethickness{.7pt}
 \qbezier(1.9,0.0)(1.9,0.787)(1.3435,1.3435)
 \qbezier(1.3435,1.3435)(0.787,1.9)(0.0,1.9)
 \qbezier(0.0,1.9)(-0.787,1.9)(-1.3435,1.3435)
 \qbezier(-1.3435,1.3435)(-1.9,0.787)(-1.9,0.0)
 \qbezier(-1.9,0.0)(-1.9,-0.787)(-1.3435,-1.3435)
 \qbezier(-1.3435,-1.3435)(-0.787,-1.9)(0.0,-1.9)
 \qbezier(0.0,-1.9)(0.787,-1.9)(1.3435,-1.3435)
 \qbezier(1.3435,-1.3435)(1.9,-0.787)(1.9,0.0)
\end{picture}}
\thicklines
\put(-15,10){\line(5,3){80}}
\put(-15,10){\line(1,4){6}}
\put(-9,34){\line(5,3){11}}
\put(43.4,65.2){\line(-5,-3){12}}
\put(65,58){\line(-3,1){21.6}}
\put(-30,-6){\line(1,0){120}}
\put(-30,-6){\line(0,1){75}}
\put(-30,69){\line(1,0){120}}
\put(90,-6){\line(0,1){75}}
\put(90,-6){\line(2,1){15}}
\put(-30,69){\line(2,1){15}}
\put(90,69){\line(2,1){15}}
\put(66,0){${\mathbb{RP}}_{n+1}$}
\put(-25,33){${\mathbb{RP}}_2$}
\end{picture}$$
In particular, it is clear that the great circles play no special r\^{o}le and 
that the space of conformal circles on $S^n$ may be identified as an open 
subset of~${\mathrm{Gr}}_3({\mathbb{R}}^{n+2})$, thus having 
dimension~$3(n-1)$. More specifically, there are three open orbits for the 
action of ${\mathrm{SO}}(n+1,1)$ on ${\mathrm{Gr}}_3({\mathbb{R}}^{n+2})$ 
according to whether the corresponding plane in  
${\mathbb{RP}}_{n+1}$ misses, touches, or slices the conformal sphere 
$S^n\hookrightarrow{\mathbb{RP}}_{n+1}$. It slices if and only if the 
quadratic form in (\ref{flat_model_conformal}) restricts to be indefinite on 
the corresponding $3$-plane in ${\mathbb{R}}^{n+2}$. The moduli space of 
conformal circles is, therefore, a homogeneous space: 
$${\mathrm{SO}}(n+1,1)/
{\mathrm{S}}({\mathrm{O}}(n-1)\times{\mathrm{O}}(2,1)).$$

The corresponding `flat model' for CR geometry arises by considering the action
of ${\mathrm{SU}}(n+1,1)$ on ${\mathbb{CP}}_{n+1}$. The unique closed orbit is
a sphere of real dimension $2n+1$, which we shall realise in the form
$$S^{2n+1}=\{[z^0,z^1,\cdots,z^n,z^\infty]\mbox{ s.t.\ }
z^0\overline{z^\infty}+z^\infty\overline{z^0}
=|z^1|^2+|z^2|^2\cdots+|z^n|^2\}.$$
${\mathrm{SU}}(n+1,1)$ acts by CR automorphisms and all such automorphisms
arise in this way (see, for example,~\cite{EG_CR} for details). Now there are
two possible geometric constructions of special curves in~S$^{2n+1}$. The first
is to intersect $S^{2n+1}$ with a linearly embedded
${\mathbb{CP}}_1\hookrightarrow{\mathbb{CP}}_{n+1}$. As verified by
Jacobowitz~\cite{J}, this gives rise to the {\em chains\/} on $S^{2n+1}$, as
defined in general CR geometry by Chern and Moser~\cite{CM}. Thus, in the flat
model, the chains may be identified as an open subset of
${\mathrm{Gr}}_2({\mathbb{C}}^{n+2})$, therefore having real
dimension $4n$ (see also~\cite{CZ}). A CR manifold $M$ is, in particular, a
contact manifold and chains are everywhere transverse to the contact
distribution $H\to M$.
There is, however, another class of `distinguished curves' in CR geometry, 
which are everywhere tangent to~$H$. To see them in the flat model, consider the 
embedding
$${\mathbb{RP}}_{n+1}\hookrightarrow{\mathbb{CP}}_{n+1}$$
induced by the inclusion ${\mathbb{R}}^{n+2}\subset{\mathbb{C}}^{n+2}$. It is
one of the two orbits for the action of ${\mathrm{SL}}(n+2,{\mathbb{R}})$ on
${\mathbb{CP}}_{n+1}$ and its intersection with the CR sphere $S^{2n+1}$ is 
the conformal sphere~$S^n$, as in~(\ref{flat_model_conformal}). Thus, we have
$$\begin{array}{ccc}{\mathbb{RP}}_{n+1}&\subset&{\mathbb{CP}}_{n+1}\\[4pt]
\cup&&\cup\\[4pt]
S^n&\subset&S^{2n+1}
\end{array}$$
and it is easily verified that $S^n$ is everywhere tangent to the CR contact
distribution $H\to S^{2n+1}$. The conformal circles in $S^n$ give distinguished
curves in $S^{2n+1}$, which may be intrinsically defined just in terms of the
CR structure on~$S^{2n+1}$, the full set of such curves being obtained under
the action of ${\mathrm{SU}}(n+1,1)$. On a general strictly-pseudoconvex CR
manifold, using the terminology of~\cite[Remark~5.3.8]{parabook}, these are the
`distinguished curves of type~(c).' We shall come back to their definition 
later in this article. Since they will be a main object of study in this 
article, we shall refer to these curves as `contact distinguished.' In the 
flat model, as above, they are closely related to conformal circles and, in 
particular, for the usual round metric on $S^{2n+1}$, they are circles. In any 
case, the moduli space of contact distinguished curves in the flat model 
is a homogeneous space ${\mathrm{SU}}(n+1,1)/S$, where
\begin{equation}\label{this_is_S}
S=\left\{C=e^{i\theta}\!\left[\begin{array}{cc}A&0\\ 0& B\end{array}\right]
\mbox{\ s.t.\ }\begin{array}{l}A\in{\mathrm{U}}(n-1)\\
B\in{\mathrm{O}}(2,1)\end{array}\enskip\det C=1\right\}\end{equation}
It has real dimension
\begin{equation}\label{real_dimension}
\big((n+2)^2-1\big)-\big(1+(n-1)^2+3-1\big)=6n-1.\end{equation}

Contact distinguished curves can be also interesting from the viewpoint of
control theory. Manifolds $M$ underlying a parabolic geometry always carry
canonical filtrations (possibly trivial, as in the conformal case) where
$T^{-1}\!M$ is a bracket-generating distribution in $TM$. Viewing the
distribution $T^{-1}\!M$ as a control distribution for a suitable system, the
bracket-generating property means that the system is {\em controllable},
i.e.~for arbitrary point $x$, each point in a neighbourhood of $x$ can be
joined with $x$ by a curve contact to $T^{-1}\!M$. The {\em contact
distinguished\/} curves are natural candidates for controlling the system
(distinguished curves have the important property that if they are contact at
one point then they are contact everywhere). On the other hand, they are
generally determined by higher jets, so there can be more distinguished curves
going in a given direction from a given point $x$. To describe a suitable
family of such curves it is therefore natural to ask if there is a connection
such that all its contact geodesics are distinguished.

In this article we set up some general machinery within parabolic geometry and
illustrate our procedures by treating, in detail, conformal geometry and
contact Legendrean geometry (an alternative real form of hypersurface CR
geometry). The conformal case is already understood by other methods
(e.g.~\cite{BE,T,Y}), which we briefly recall before using symmetry algebras
and the tractor connection to rederive the conformal results in a way that
generalises to all parabolic geometries. The key here is the characterisation
of unparameterised distinguished curves, due to Doubrov and
\v{Z}\'adn\'{\i}k~\cite{DZ}, in terms of the Cartan connection. We rephrase
their characterisation in terms of tractor connections (and, in an appendix,
derive the equations of Tod~\cite{T} for conformal circles by tractor methods).
An important point about such methods is that they are manifestly invariant.
Otherwise, one has to check the invariance of the resulting equations by
tedious and unilluminating calculation. In fact, we do not need the equations
themselves in order to understand when (contact) geodesics for a Weyl
connection within a parabolic class are distinguished. Much of this article is
devoted to carrying out this procedure in the context of contact Legendrean
geometry, culminating in Theorem~\ref{contact_Legendrean_theorem} and its
corollary. There are corresponding results in the CR setting, namely
Theorem~\ref{CR_theorem} and its corollary. In both cases, it is interesting to
note that there are non-trivial examples, which may be gleaned from articles of
Loboda~\cite{L} and, more recently, Kruglikov~\cite{K} and
Doubrov-Medvedev-The~\cite{DMT,DMT_CR}. In the CR setting, the geometry we 
identify was already considered by Leitner~\cite{Leitner} who explains how to 
derive them from K\"ahler-Einstein metrics.

\subsection*{Acknowledgements}
We would like to thank Katharina Neusser for most helpful discussions leading
to our deciphering the conformal circle equation via standard tractors. 

We would also like to acknowledge that our understanding of contact Legendrean
geometry was significantly enriched by a joint project of the second author
with Gianni Manno, Katja Sagerschnig, and Josef \v{S}ilhan. In particular, the 
calculations behind the examples of \S\ref{contact_legendrean_examples} use 
formul{\ae} and {\sc Maple} programs derived in this project.

Conversations with Rod Gover were very useful in relating our results to the 
work of Leitner~\cite{Leitner} and \v{C}ap-Gover~\cite{CG}.

Finally, we thank the referee for several useful suggestions.

\section{Conformal geometry}\label{conformal_geometry}
Fix a Riemannian conformal structure on~$M$. We need to define conformal
circles on $M$ as unparameterised curves. Since it is our aim to compare these
curves with metric geodesics, it is convenient to follow Tod~\cite{T} in
selecting a {\em background metric\/} $g_{ab}$ from the conformal class and
write all differential equations with respect to this metric and its
Levi-Civita connection $\nabla_a$, whilst checking that these equations
do not depend on our choice of~$g_{ab}$.

Let $\gamma\hookrightarrow M$ be a smooth oriented, embedded curve. In the
presence of~$g_{ab}$, we shall write $U^a$ for the unique unit length vector
field defined along and tangent to~$\gamma$ in the direction of its
orientation. Evidently, if $g_{ab}$ is replaced by
$\widehat{g}_{ab}=\Omega^2g_{ab}$, then $U^a$ is replaced by
$\widehat{U}^a=\Omega^{-1}U^a$. We say that $U^a$ has {\em conformal
weight\/}~$-1$, equivalently that $U_a\equiv g_{ab}U^b$ has conformal
weight~$1$. Again in the presence of a metric~$g_{ab}$, let $\partial\equiv
U^a\nabla_a$ denote the {\em directional derivative\/} along~$\gamma$, where
$\nabla_a$ is the Levi-Civita connection of~$g_{ab}$. The directional
derivative may be applied to any tensor defined along $\gamma$ and, in
particular, to the vector field~$U^a$. For a given metric $g_{ab}$, we shall
refer to $U^a$ as the {\em velocity\/} along $\gamma$ and 
$$C^a\equiv\partial U^a\quad\mbox{as the {\em acceleration} along}~\gamma.$$
Then $\gamma$ is a {\em geodesic\/} for $g_{ab}$ if and only if~$C^a\equiv 0$.
Notice that, along any curve, since $0=\partial(U^aU_a)=2U^aC_a$, the
acceleration is orthogonal to the velocity. 
\begin{lemma}\label{ode} If
$\omega$ is a smooth $1$-form on $M$, then locally it is always possible to
find a smooth function $f$ on $M$ such that $\omega=df$ along~$\gamma$.
\end{lemma}
\begin{proof} In $2$-dimensions we may choose local co\"ordinates $(x,y)$ so
that $\gamma$ is the $x$-axis. Then
$$\omega|_\gamma=\alpha(x)\,dx+\beta(x)\,dy$$
and we may take
$$f(x,y)=\int_\gamma\alpha(x)\,dx+\beta(x)\,y.$$
Away from a tubular neighbourhood of~$\gamma$, we may extend $f$
arbitrarily by a partition of unity. The $n$-dimensional case is similar.
\end{proof}
\begin{lemma}\label{find_metric} Locally, it
is always possible to find a metric in the given
conformal class for which $\gamma$ is a geodesic.
\end{lemma}
\begin{proof}
If we replace $g_{ab}$ by $\widehat{g}_{ab}=\Omega^2g_{ab}$, then the
Levi-Civita connection $\nabla_a$ acting on an arbitrary $1$-form $\phi_b$, is
replaced by $\widehat{\nabla}_a$ acting on $\phi_b$ according to the formula
$$\widehat\nabla_a\phi_b=\nabla_a\phi_b
-\Upsilon_a\phi_b-\Upsilon_b\phi_a+\Upsilon^c\phi_cg_{ab},$$
where $\Upsilon_a=\nabla_a\log\Omega$. 
Therefore, 
$$\widehat{C}_b=\widehat{U}^a\widehat\nabla_a\widehat{U}_b
=\Omega^{-1}U^a\widehat\nabla_a(\Omega U_b)
=C_b-\Upsilon_b+U^c\Upsilon_cU_b$$
and we aim to enforce $\widehat{C}_b=0$ along $\gamma$ by a suitable choice
of~$\Omega$. By Lemma~\ref{ode} we may choose $\Omega$ positive so that 
$\nabla_a\log\Omega=C_a$ along~$\gamma$. Then $\widehat{C}_b=U^c\Upsilon_cU_b$
along $\gamma$. But $\widehat{C}^b$ is also orthogonal to $\gamma$ and 
hence must vanish.
\end{proof}
We shall define the Riemann curvature tensor $R_{abcd}$ by
$$(\nabla_a\nabla_b-\nabla_b\nabla_a)X^c=R_{ab}{}^c{}_dX^d,$$
using the metric to lower the index $c$ in the usual way. The Schouten tensor
is the symmetric tensor $\Rho_{ab}$ such that
\begin{equation}\label{schouten}R_{abcd}
=W_{abcd}+\Rho_{ac}g_{bd}-\Rho_{bc}g_{ad}-\Rho_{ad}g_{bc}+\Rho_{bd}g_{ac}
\end{equation}
where $W_{ab}{}^a{}_d=0$. 
Following Tod~\cite{T}, we say that an unparameterised curve
$\gamma\hookrightarrow M$ is a {\em conformal circle\/} if and only if
\begin{equation}\label{cc}
\partial C_a=\Rho_{ab}U^b-(C^bC_b+\Rho_{bc}U^bU^c)U_a\end{equation}
along~$\gamma$. Defined this way, it requires a computation to check conformal 
invariance: use, for example, the formul{\ae} in~\cite{E} and that
$$\textstyle\widehat\Rho_{ab}=\Rho_{ab}
-\nabla_a\Upsilon_b+\Upsilon_a\Upsilon_b-\frac12\Upsilon^c\Upsilon_cg_{ab}.$$
In an appendix, however, we shall derive (\ref{cc}) by tractor calculus,
thereby ensuring its conformal invariance. Later in this section, we shall use
tractor calculus to avoid the general form of~(\ref{cc}), deriving this
equation only when it so happens that $\gamma$ is a geodesic. This special case
is all we need in order to prove the following.
\begin{thm}\label{one}
A curve $\gamma\hookrightarrow M$ is a conformal circle if and only
if there is a metric $g_{ab}$ in the conformal class for which $\gamma$ is a
geodesic and so that\/ $\Rho_{ab}U^b=\Rho_{bc}U^bU^cU_a$ along~$\gamma$.
\end{thm}
\begin{proof}
If $\gamma$ is a geodesic and $\Rho_{ab}U^b=\Rho_{bc}U^bU^cU_a$ along~$\gamma$,
then (\ref{cc}) is satisfied and $\gamma$ is a conformal circle. Conversely, if
we choose by Lemma~\ref{find_metric} a metric in the conformal class for which
$\gamma$ is a geodesic, then (\ref{cc}) implies that
$\Rho_{ab}U^b=\Rho_{bc}U^bU^cU_a$, as required.\end{proof}
\begin{cor}
The geodesics of a Riemannian metric are all conformal circles if and 
only if the metric is Einstein.
\end{cor}
\begin{proof} To say that $\Rho_{ab}U^b=\Rho_{bc}U^bU^cU_a$ is exactly to say
that $U_b$ is an eigenvector of the endomorphism $\Rho_a{}^b$ (since if
$\Rho_a{}^bU_b=\lambda U_a$, then it must be that $\lambda=U^a\Rho_a{}^bU_b
=\Rho_{bc}U^bU^c$).
All vectors are eigenvectors if and only if $\Rho_a{}^b=\lambda\delta_a{}^b$,
where $\delta_a{}^b$ is the identity matrix.
\end{proof}
\noindent{\bf Remark}\enskip Smooth curves in a conformal manifold come
equipped with a class of preferred parameterisations, well-defined up to a
projective freedom. If we ask, not only that a geodesic $\gamma$ of a
Riemannian metric be a conformal circle, but also that its arc-length
parameterisation be preferred, then we require, in addition, that
$$\textstyle U^a\partial C_a=3(U^aC_a)^2-\frac32C^aC_a-\Rho_{ab}U^aU^b$$
and it follows that $\Rho_{ab}U^b=0$ along~$\gamma$. This is what is proved
in~\cite{BE}. We obtain a different corollary, namely that for all 
{\em arc-length parameterised\/} geodesics of a metric to be {\em projectively
parameterised\/} conformal circles, it is necessary and sufficient that the
metric be Ricci-flat. A good example in four dimensions is a Calabi-Yau metric 
on a K3 surface.

\medskip\noindent{\bf Remark}\enskip This characterisation of {\em
parameterised\/} distinguished curves holds true in any parabolic
geometry~\cite[Proposition~3.7]{Z}.

\medskip
Following Doubrov and \v{Z}\'adn\'{\i}k~\cite{DZ}, we may prove
Theorem~\ref{one} above without firstly having to derive the differential
equation~(\ref{cc}). The main advantage of this alternative proof is that it
applies to all parabolic geometries. To implement it, all one needs is the
symmetry algebra of an unparameterised curve $L$ in the flat model~$G/P$ and
the Cartan connection, regarded as a ${\mathfrak{g}}$-valued $1$-form $\omega$
on the total space of the Cartan bundle ${\mathcal{G}}\to M$. The symmetry
algebra ${\mathrm{Sym}}L$ of $L$ is defined to be a certain subalgebra of
${\mathfrak{g}}$ realised as vector fields on $G/P$ induced by the action of
$G$ on~$G/P$. Specifically,
\begin{equation}\label{SymL}
{\mathrm{Sym}}L\equiv\{X\in{\mathfrak{g}}\mid
\mbox{$X$ is tangent to $L$ at $p$ for all $p\in L$}\}.\end{equation}
It is evident that ${\mathrm{Sym}}L$ is non-trivial if and only if $L$ is
homogeneous. The notion of conformal circles generalises to any parabolic
geometry. They are called {\em distinguished curves\/} and
\cite[Proposition~1]{DZ} shows that an {\em unparameterised\/} curve
$\gamma\hookrightarrow M$ is distinguished if and only if there is a smooth
section $\gamma\to{\mathcal{G}}$ of the Cartan bundle ${\mathcal{G}}\to M$ so
that $\gamma^*\omega$ takes values in 
${\mathrm{Sym}}L\subseteq{\mathfrak{g}}$ (this is for distinguished curves 
{\em modelled\/} on $L\hookrightarrow G/P$). If preferred, one could start with 
this as the definition of distinguished curves in a Cartan geometry.

In what follows we shall unpack this characterisation of distinguished curves
in conformal geometry, supposing, for simplicity, that $\gamma\hookrightarrow
M$ is already a geodesic for a chosen metric in the conformal class.  By
Lemma~\ref{find_metric} this is no restriction on $\gamma$ and, in any case,
will evidently be sufficient in providing a route to Theorem~\ref{one}.  The
object of this exercise is to avoid (\ref{cc}) since this is more difficult to
derive.  For contact distinguished curves in CR geometry, for example, we do
not know, nor need to know, the equation corresponding to (\ref{cc}) in order
to establish, in Section~\ref{CR_geometry}, the analogue of Theorem~\ref{one}.
More precisely, our upcoming Theorem~\ref{distinguished} will soon translate
the Doubrov-\v{Z}\'adn\'{\i}k characterisation \cite{DZ} of unparameterised
distinguished curves for a general parabolic geometry into the language of
tractors.  But, to apply Theorem~\ref{distinguished}, we shall need to know the
symmetry algebra (\ref{SymL}) of a model curve in~$G/P$.  In the conformal case
it is easiest to suppose the initial acceleration vanishes.  For completeness,
in an appendix we shall use Theorem~\ref{distinguished} to derive equation
(\ref{cc}) for conformal circles in general.

When $\gamma\hookrightarrow M$ is already a geodesic, the model curve 
in~${\mathbb{R}}^n$ will be a straight line
\begin{equation}\label{line}{\mathbb{R}}\ni t\mapsto 2tU^b\end{equation}
for $U^b$ of unit length (the seemingly spurious factor of $2$ being included 
here so that this straight line can be seen as the limiting case of a 
circle~(\ref{circle}) when, in the Appendix, we include acceleration). Hence, 
the first order of business is to compute the conformal symmetry algebra 
of~(\ref{line}). The Lie algebra of conformal Killing fields on 
${\mathbb{R}}^n$ is
\begin{equation}\label{conformal_Killing_fields}\textstyle 
(X^b-F^b{}_cx^c+\lambda x^b-Y_cx^cx^b+\frac12x_cx^cY^b)
\,\partial/\partial x^b\end{equation}
for constant tensors $X^b,F^b{}_c,\lambda,Y_b$ with $F_{bc}$ skew. We see
if such a field is everywhere tangent to~(\ref{line}), arriving at the
following conclusion. 
\begin{lemma}\label{symmetry_algebra_of_a_line}
The symmetry algebra of the line \eqref{line} consists of those
fields \eqref{conformal_Killing_fields} satisfying the following linear
constraints
\begin{equation}\label{sym_alg_of_line}
X^b=fU^b\qquad F^b{}_cU^c=0\qquad Y^b=hU^b
\end{equation}
where $f$ and $h$ are arbitrary.  
\end{lemma}
Note that this symmetry algebra has dimension $\big((n-1)(n-2)/2\big)+3$, as
will remain true when we include acceleration (in the Appendix).

To unpack the Doubrov-\v{Z}\'adn\'{\i}k characterisation, we may use tractors
instead of the Cartan connection. Recall that, in the presence of a
metric~$g_{ab}$, the standard tractor connection~\cite{BEG} is given by
\begin{equation}\label{conformal_tractor_connection}
\nabla_a\left[\begin{array}c\sigma\\ \mu_b\\
\rho\end{array}\right]= \left[\begin{array}c\nabla_a\sigma-\mu_a\\
\nabla_a\mu_b+g_{ab}\rho+\Rho_{ab}\sigma\\ 
\nabla_a\rho-\Rho_a{}^b\mu_b\end{array}\right]\end{equation}
on the bundle ${\mathcal{T}}=\Wedge^0[1]+\Wedge^1[1]+\Wedge^0[-1]$ and the 
invariant inner product is
$$\left\langle 
\left[\begin{array}c\sigma\\ \mu_b\\ \rho\end{array}\right],
\left[\begin{array}c\tilde\sigma\\ \tilde\mu_b\\ \tilde\rho\end{array}\right]
\right\rangle=
\sigma\tilde\rho+\mu^b\tilde\mu_b+\rho\tilde\sigma.$$
\begin{lemma}The general $\langle\enskip,\enskip\rangle$-preserving 
endomorphism of ${\mathcal{T}}$ has the form
\begin{equation}\label{endomorphisms}
\Phi\left[\begin{array}c\sigma\\ \mu_b\\ \rho\end{array}\right]
=\left[\begin{array}c X^b\mu_b-\lambda\sigma\\
Y_b\sigma+F_b{}^c\mu_c-X_b\rho\\ 
\lambda\rho-Y^b\mu_b\end{array}\right]\end{equation}
for unweighted tensors $X^b,F^b{}_c,\lambda,Y_b$ with $F_{bc}$ skew.
\end{lemma}
\begin{proof} An elementary verification.\end{proof}
In fact, the $\langle\enskip,\enskip\rangle$-preserving endomorphisms of
${\mathcal{T}}$ are, by definition, sections of the {\em adjoint tractor
bundle\/}~${\mathcal{A}}$, a notion that makes sense in any parabolic geometry:
it is the bundle induced from the Cartan bundle ${\mathcal{G}}\to M$ by the
Adjoint representation of $G$ on~${\mathfrak{g}}$. In particular, the symmetry
algebra ${\mathrm{Sym}}L\subset{\mathfrak{g}}$ of a homogeneous curve 
$L\subset G/P$ induces a collection of preferred subspaces of ${\mathcal{A}}$,
the conjugates of ${\mathrm{Sym}}L$ (under the Adjoint action of $G$
on~${\mathcal{A}}$). Translating the Doubrov-\v{Z}\'adn\'{\i}k result~\cite{DZ}
into the language of tractors gives the following.
\begin{thm}\label{distinguished} In order that $\gamma\hookrightarrow M$ be an
unparameterised distinguished curve modelled on $L\hookrightarrow G/P$, it is
necessary and sufficient that along $\gamma$ there be a subbundle
${\mathcal{S}}\subset{\mathcal{A}}|_\gamma$ whose fibres are everywhere
conjugate to ${\mathrm{Sym}}L\subset{\mathfrak{g}}$ and such that
${\mathcal{S}}$ is preserved by the tractor connection along~$\gamma$.
\end{thm}
\paragraph{\bf Remarks} Firstly, note that this is a manifestly invariant
formulation. Secondly, there is a canonical projection ${\mathcal{A}}\to TM$
and if $\gamma\hookrightarrow M$ is to be a distinguished curve in accordance
with Theorem~\ref{distinguished}, then the image of 
${\mathcal{S}}\subset{\mathcal{A}}|_\gamma$ in $TM|_\gamma$ is
forced to be the tangent bundle of~$\gamma$.

\medskip
In the particular case of a Riemannian manifold $M$ with metric~$g_{ab}$, 
we conclude that a {\em geodesic\/} $\gamma\hookrightarrow M$
is a conformal circle if and only if we can find a subbundle ${\mathcal{S}}$ of
the endomorphisms of ${\mathcal{T}}$ along~$\gamma$, having the form
(\ref{endomorphisms}) and such that
\begin{itemize}
\item $X^a$ is tangent to $\gamma$,
\item fibrewise, the subbundle ${\mathcal{S}}$ has the 
form~(\ref{sym_alg_of_line}),
\item ${\mathcal{S}}$ is preserved by the tractor connection $\partial\equiv 
U^a\nabla_a$ along~$\gamma$.
\end{itemize}
To proceed, it is useful to reformulate the conditions (\ref{sym_alg_of_line})
as follows. 
\begin{lemma}\label{reformulate}
In order that an endomorphism \eqref{endomorphisms}
satisfy~\eqref{sym_alg_of_line}, for some unit vector field $U^a$, it is
necessary and sufficient that
\begin{equation}\label{nec_and_suff}
\Phi\!\left[\!\begin{array}c0\\ 0\\ 1\end{array}\!\right]
\!=\!\left[\!\!\begin{array}c 0\\
-fU_b\\ 
\lambda\end{array}\!\!\right]\qquad
\Phi\left[\!\begin{array}c0\\ U_b\\ 0\end{array}\!\right]
\!=\!\left[\!\!\begin{array}c f\\
0\\ 
-h\end{array}\!\!\right]\qquad
\Phi\!\left[\!\begin{array}c 1\\ 0\\ 0\end{array}\!\right]
\!=\!\left[\!\!\begin{array}c -\lambda\\
hU_b\\ 
0\end{array}\!\!\right]\!.
\end{equation}
\end{lemma} 
\begin{proof}
If (\ref{sym_alg_of_line}) are satisfied, then all equations of
(\ref{nec_and_suff}) hold with $X^b=fU^b$ and $Y^b=hU^b$. Conversely, if
(\ref{nec_and_suff}) hold, then the first two equations determine $f$, $U^a$,
$\lambda$, and $h$. {From} (\ref{endomorphisms}), the first two constraints
from (\ref{sym_alg_of_line}) are manifest and we also discover that $Y_b=hU_b$,
which is the final constraint from (\ref{sym_alg_of_line}).\end{proof} 

\begin{proof}[Proof of Theorem~\ref{one}]
According to Theorem~\ref{distinguished}, it suffices to show that the 
conditions (\ref{nec_and_suff}) on a tractor
endomorphism $\Phi$ are preserved by the tractor connection along a geodesic
$\gamma$ of the metric $g_{ab}$ if and only if
\begin{equation}\label{Schouten_along_geodesic}
\Rho_{ab}U^b=\Rho_{bc}U^bU^cU_a\qquad\mbox{along}\enskip\gamma.\end{equation}
This is a straightforward verification as follows.
{From}~(\ref{conformal_tractor_connection}), the tractor connection
$\partial=U^a\nabla_a$ along $\gamma$ is given by
\begin{equation}\label{directional_derivative}
\partial\left[\begin{array}c\sigma\\ \mu_b\\ \rho\end{array}\right]= 
\left[\begin{array}c\partial\sigma-U^a\mu_a\\
\partial\mu_b+U_b\rho+U^a\Rho_{ab}\sigma\\ 
\partial\rho-U^a\Rho_a{}^b\mu_b\end{array}\right].\end{equation} Bearing in
mind that $\partial U^b=0$ (since $\gamma$ is a geodesic) the Leibniz rule now
calculates the effect of $\partial\Phi$ and, firstly, we find from
(\ref{nec_and_suff}) that
$$(\partial\Phi)\!\left[\begin{array}c0\\ 0\\ 1\end{array}\right]
=\partial\!\left[\begin{array}c 0\\
\!\!-fU_b\!\\ \lambda\end{array}\right]-
\Phi\partial\!
\left[\begin{array}c0\\ 0\\ 1\end{array}\right]
=\left[\begin{array}c0\\ 
-(\partial f-\lambda)U_b\\
\!\partial\lambda+h+\Rho_{ab}U^aU^b
\end{array}\right]\!,$$
which has the same form with $f$ and $\lambda$ replaced by 
$$\tilde{f}=\partial f-\lambda\quad\mbox{and}\quad
\tilde\lambda=\partial\lambda+h+f\Rho_{ab}U^aU^b,$$
respectively. Next we should compute
$$(\partial\Phi)\left[\begin{array}c0\\ U_b\\ 0\end{array}\right]
=\partial\Big(\Phi\left[\begin{array}c0\\ U_b\\ 0\end{array}\right]\Big)
-\Phi\partial\left[\begin{array}c0\\ U_b\\ 0\end{array}\right]\!,$$
and, from (\ref{nec_and_suff}) and~(\ref{directional_derivative}), we find
$$\partial\Big(\Phi\left[\begin{array}c0\\ U_b\\ 0\end{array}\right]\Big)
=\partial\left[\begin{array}c f\\
0\\ 
-h\end{array}\right]
=\left[\begin{array}c\partial f\\ 
-hU_b+fU^a\Rho_{ab}\\ 
-\partial h\end{array}\right]$$
whilst
$$\Phi\partial\left[\begin{array}c0\\ U_b\\ 0\end{array}\right]
=\Phi\left[\begin{array}c-1\\
0\\ 
-\Rho_{ab}U^aU^b\end{array}\right]
=\left[\begin{array}c \lambda\\
(f\Rho_{ac}U^aU^c-h)U_b\\ 
-\lambda\Rho_{bc}U^bU^c\end{array}\right].$$
Therefore, 
$$(\partial\Phi)\left[\begin{array}c0\\ U_b\\ 0\end{array}\right]
=\left[\begin{array}c\partial f-\lambda\\ 
f(\Rho_{bc}U^c-\Rho_{ac}U^aU^cU_b)\\ 
-\partial h+\lambda\Rho_{bc}U^bU^c\end{array}\right],$$
which has the form required by (\ref{nec_and_suff}), with $f$ and $h$ replaced 
by 
$$\tilde{f}=\partial f-\lambda\quad\mbox{and}\quad
\tilde{h}=\partial h-\lambda\Rho_{bc}U^bU^c,$$
respectively, if and only if (\ref{Schouten_along_geodesic}) holds. Finally, 
we need to compute
$$(\partial\Phi)\left[\begin{array}c1\\ 0\\ 0\end{array}\right]
=\partial\Big(\Phi\left[\begin{array}c1\\ 0\\ 0\end{array}\right]\Big)
-\Phi\partial\left[\begin{array}c1\\ 0\\ 0\end{array}\right]\!.$$
Well, from (\ref{nec_and_suff}), we find 
$$\partial\Big(\Phi\left[\begin{array}c1\\ 0\\ 0\end{array}\right]\Big)
=\partial\left[\begin{array}c -\lambda\\
hU_b\\ 0\end{array}\right]
=\left[\begin{array}c -\partial\lambda-h\\
(\partial h)U_b-\lambda\Rho_{ab}U^a\\ 
-h\Rho_{ac}U^aU^c\end{array}\right]$$
and
$$\Phi\partial\left[\begin{array}c1\\ 0\\ 0\end{array}\right]
=\Phi\left[\begin{array}c 0\\
U^a\Rho_{ab}\\ 
0\end{array}\right]
=\Phi\left[\begin{array}c 0\\
\Rho_{ac}U^aU^cU_b\\ 
0\end{array}\right],$$
where we have used (\ref{Schouten_along_geodesic}) to
substitute~$U^a\Rho_{ab}=\Rho_{ac}U^aU^cU_b$. From here, 
$$\Phi\partial\left[\begin{array}c1\\ 0\\ 0\end{array}\right]
=\left[\begin{array}c f\Rho_{ac}U^aU^c\\ 0\\ 
-h\Rho_{ac}U^aU^c\end{array}\right]$$
and so, finally, again using (\ref{Schouten_along_geodesic}) to substitute
$\Rho_{ac}U^aU^cU_b$ for~$\Rho_{ab}U^a$,
$$(\partial\Phi)\left[\begin{array}c1\\ 0\\ 0\end{array}\right]
=\left[\begin{array}c -\partial\lambda-h-f\Rho_{ac}U^aU^c\\
(\partial h)U_b-\lambda\Rho_{ab}U^a\\ 
0\end{array}\right]
=\left[\begin{array}c \tilde{\lambda}\\
\tilde{h}U_b\\ 
0\end{array}\right],$$
as required.\end{proof}

\section{The symmetry algebra}
In \S\ref{conformal_geometry} we defined, by means of~(\ref{SymL}), the
symmetry algebra ${\mathrm{Sym}}L$ for any curve $L\hookrightarrow G/P$.
Starting with any $V\in{\mathfrak{g}}$, we may consider the curve $L_V$ in
$G/P$ obtained as the image of $t\mapsto\exp(tV)\in G$. If
$V\not\in{\mathfrak{p}}$, this curve is non-trivial and homogeneous. By
construction, we have $V\in{\mathrm{Sym}}L_V$ but the full symmetry algebra
${\mathrm{Sym}}L_V$ is generally bigger than $\langle V\rangle$, the span
of~$V$. We expect a completely algebraic procedure to obtain
${\mathrm{Sym}}L_V$ from $V\in{\mathfrak{g}}\setminus{\mathfrak{p}}$. It may be
given as follows.
\begin{lemma}[from~\cite{DKR}]\label{dkr_lemma} Consider the series of 
subalgebras 
$${\mathfrak{p}}={\mathfrak{a}}_0\supseteq{\mathfrak{a}}_1
\supseteq{\mathfrak{a}}_2\supseteq\cdots$$
defined inductively by 
$${\mathfrak{a}}_{\ell+1}\equiv\{X\in{\mathfrak{a}}_\ell\mid 
[X,V]\in{\mathfrak{a}}_\ell+\langle V\rangle\}$$
and let ${\mathfrak{a}}_\infty=\cap_{k=0}^\infty\,{\mathfrak{a}}_\ell$. Then
${\mathrm{Sym}}L_V={\mathfrak{a}}_\infty+\langle V\rangle$.
\end{lemma}
In Lemma~\ref{symmetry_algebra_of_a_line} we computed the symmetry algebra of a
straight line in ${\mathbb{R}}^n$ directly from the definition (\ref{SymL}) and
in Lemma~\ref{symmetry_algebra_of_a_circle} we similarly and directly compute
the symmetry algebra of a circle in~${\mathbb{R}}^n$. These computations may be
circumvented by Lemma~\ref{dkr_lemma} provided we have a convenient realisation
of the Lie algebras ${\mathfrak{g}}\supset{\mathfrak{p}}$. 
In the conformal case, for example, we may identify the conformal Killing 
fields (\ref{conformal_Killing_fields}) with matrices 
$$\left[\begin{array}{ccc}\lambda&-Y_c&0\\
-X^b&F^b{}_c&Y^b\\
0&X_c&-\lambda
\end{array}\right]\in{\mathfrak{so}}(n+1,1)$$ 
and the parabolic subalgebra ${\mathfrak{p}}$ as matrices of the form 
$$\left[\begin{array}{ccc}\lambda&-Y_c&0\\
0&F^b{}_c&Y^b\\
0&0&-\lambda
\end{array}\right].$$ 
For our element $V\in{\mathfrak{g}}\setminus{\mathfrak{p}}$, let us take 
$$\left[\begin{array}{ccc}0&0&0\\
-U^b&0&0\\
0&U_c&0\end{array}\right],\enskip\mbox{such that}\enskip U^b U_b=1.$$
Then one readily verifies that
$${\mathfrak{a}}_1=
\left\{\left[\begin{array}{ccc}\lambda&-Y_c&0\\
0&F^b{}_c&Y^b\\
0&0&-\lambda
\end{array}\right]\mbox{\ s.t.\ }F^b{}_cU^c=0\right\},$$
and hence that
$${\mathfrak{a}}_\infty={\mathfrak{a}}_2=
\left\{\left[\begin{array}{ccc}\lambda&-hU_c&0\\
0&F^b{}_c&hU^b\\
0&0&-\lambda
\end{array}\right]\mbox{\ s.t.\ }F^b{}_cU^c=0\right\},$$
whence, by Lemma~\ref{dkr_lemma}, the full symmetry algebra is
\begin{equation}\label{conf_sym}{\mathfrak{s}}=
\left\{\left[\begin{array}{ccc}\lambda&-hU_c&0\\
-fU^b&F^b{}_c&hU^b\\
0&fU_c&-\lambda
\end{array}\right]\mbox{\ s.t.\ }F^b{}_cU^c=0\right\}\end{equation}
in agreement with (\ref{sym_alg_of_line}).

In \S\ref{contact_Legendrean_geometry}, we shall use
Lemma~\ref{dkr_lemma} to compute a different symmetry algebra, closely related
to the contact distinguished curves in CR geometry, as described
in~\S\ref{introduction}.

\subsection{The symmetry algebra in parabolic geometry}\label{para_sym}
The purpose of this subsection is to specialise the result of
Lemma~\ref{dkr_lemma} to the case that the homogeneous space $M=G/P$ is a flat
parabolic geometry and the curve $L\hookrightarrow G/P$ starts off, using the
notation from~\cite{parabook}, in a direction from $T^{-1}\!M\subseteq TM$.
All the examples in this article are modelled in this way and the general
algorithm given in Lemma~\ref{dkr_lemma} is especially congenial in this case.
To this end, suppose ${\mathfrak{g}}$ is semi\-simple,
${\mathfrak{p}}\subset{\mathfrak{g}}$ is parabolic, and
$V\in{\mathfrak{g}}\setminus{\mathfrak{p}}$ is nicely-positioned with respect
to a parabolic grading.  Specifically, let us suppose that ${\mathfrak{g}}$ is
$|k|$-graded:
$${\mathfrak{g}}={\mathfrak{g}}_{-k}\oplus\cdots\oplus{\mathfrak{g}}_{-1}
\oplus\underbrace{{\mathfrak{g}}_0\oplus{\mathfrak{g}}_1\oplus\cdots\oplus
{\mathfrak{g}}_k}_{\mbox{\normalsize$={\mathfrak{p}}$}}$$
as in~\cite{parabook} and that we start with
$V\in{\mathfrak{g}}_{-1}$. Recall that Lemma~\ref{dkr_lemma} defines
$${\mathfrak{a}}_1
\equiv\{X\in{\mathfrak{p}}\mid[X,V]\in{\mathfrak{p}}+\langle V\rangle\}.$$
Evidently, this is only a constraint on the ${\mathfrak{g}}_0$-component
of~$X$. Then, by induction, the subalgebra ${\mathfrak{a}}_{\ell+1}$ constrains
only the ${\mathfrak{g}}_\ell$-component. In particular, we can see when
matters stabilise; specifically ${\mathfrak{a}}_\infty={\mathfrak{a}}_{k+1}$.

In the conformal case for example, the Lie algebra
${\mathfrak{g}}={\mathfrak{so}}(n+1,1)$ is $|1|$-graded and, employing Dynkin
diagram notation from~\cite{beastwood},
$$\begin{array}{ccccc}{\mathfrak{so}}(n+1,1)&=&\oooodots{0}{1}{0}{0}\\
\|\\[-16pt]
{\mathfrak{g}}&=&{}\xooodots{0}{1}{0}{0}&\!\oplus
\begin{array}{c}\xooodots{-1}{0}{1}{0}\\
\oplus\\    \xooodots{0}{0}{0}{0}
\end{array}\oplus&\!\xooodots{-2}{1}{0}{0}\\
&&\|&\|&\|\\
&&{\mathfrak{g}}_{-1}&{\mathfrak{g}}_0&{\mathfrak{g}}_1
\end{array}$$
so ${\mathfrak{a}}_\infty={\mathfrak{a}}_2$ is confirmed and we can locate the
different pieces of the symmetry algebra~(\ref{conf_sym}), namely
$$\begin{array}{ccccc}
{\mathfrak{g}}_{-1}&\oplus&{\mathfrak{g}}_0&\oplus&{\mathfrak{g}}_1\\
\inn&&\inn&&\inn\\
fU^b&&\left(\!\!\begin{array}{c}
F_{bc}\\ \lambda\end{array}\!\!\right)&&hU_c\makebox[0pt][l]{.}
\end{array}$$

\section{Contact Legendrean geometry}\label{contact_Legendrean_geometry}
This geometry is an alternative real form of strictly-pseudoconvex CR geometry.
The flat model is a homogeneous space for ${\mathrm{SL}}(n+2,{\mathbb{R}})$
instead of the sphere $S^{2n+1}$ as a homogeneous space for
${\mathrm{SU}}(n+1,1)$. Specifically, it is the flag manifold
$${\mathbb{F}}_{1,n+1}({\mathbb{R}}^{n+2})
=\{{\mathcal{L}}\subset{\mathcal{H}}\subset{\mathbb{R}}^{n+2}\mid
\dim{\mathcal{L}}=1,\dim{\mathcal{H}}=n+1\}$$
and it is convenient to view elements of ${\mathfrak{sl}}(n+2,{\mathbb{R}})$ in 
blocks:
\begin{equation}\label{block_matrix}
\left[\begin{array}{ccc}a&Z_\beta&b\\
X^\alpha&C^\alpha{}_\beta&W^\alpha\\
d&Y_\beta&e\end{array}\right],\enskip\mbox{where}\enskip
a+C^\alpha{}_\alpha+e=0.\end{equation}
Let us take ${\mathfrak{p}}$ to be the block upper triangular
matrices, more precisely the subalgebra comprising elements of the form
$$\left[\begin{array}{ccc}a&Z_\beta&b\\
0&C^\alpha{}_\beta&W^\alpha\\
0&0&e\end{array}\right].$$
For our element in ${\mathfrak{g}}\setminus{\mathfrak{p}}$ let us take
$$\left[\begin{array}{ccc}0&0&0\\
U^\alpha&0&0\\
0&V_\alpha&0\end{array}\right],\enskip\mbox{such that}
\enskip U^\alpha V_\alpha=1.$$
Then one readily verifies that 
$${\mathfrak{a}}_1=\left\{\left[\begin{array}{ccc}a&Z_\alpha&b\\
0&C^\alpha{}_\beta&W^\alpha\\
0&0&e\end{array}\right]\mbox{\ s.t.\ }\begin{array}{l}
C^\alpha{}_\beta U^\beta=\frac{a+e}2U^\alpha\\[4pt]
V_\alpha C^\alpha{}_\beta=\frac{a+e}2V_\beta
\end{array}\right\},$$
that
$${\mathfrak{a}}_2=\left\{\left[\begin{array}{ccc}a&hV_\beta&b\\
0&C^\alpha{}_\beta&hU^\alpha\\
0&0&e\end{array}\right]\mbox{\ s.t.\ }\begin{array}{l}
C^\alpha{}_\beta U^\beta=\frac{a+e}2U^\alpha\\[4pt]
V_\alpha C^\alpha{}_\beta=\frac{a+e}2V_\beta
\end{array}\right\},$$
and finally that
$${\mathfrak{a}}_\infty={\mathfrak{a}}_3
=\left\{\left[\begin{array}{ccc}a&hV_\beta&0\\
0&C^\alpha{}_\beta&hU^\alpha\\
0&0&e\end{array}\right]\mbox{\ s.t.\ }\begin{array}{l}
C^\alpha{}_\beta U^\beta=\frac{a+e}2U^\alpha\\[4pt]
V_\alpha C^\alpha{}_\beta=\frac{a+e}2V_\beta
\end{array}\right\}$$
whence the full symmetry algebra is
\begin{equation}\label{full_symmetry_algebra}
{\mathfrak{s}}=\left\{\left[\begin{array}{ccc}a&hV_\beta&0\\
fU^\alpha&C^\alpha{}_\beta&hU^\alpha\\
0&fV_\beta&e\end{array}\right]\mbox{\ s.t.\ }\begin{array}{l}
C^\alpha{}_\beta U^\beta=\frac{a+e}2U^\alpha\\[4pt]
V_\alpha C^\alpha{}_\beta=\frac{a+e}2V_\beta
\end{array}\right\}.\end{equation}
As a consistency check, notice that 
$$\dim{\mathfrak{s}}=\begin{tabular}[t]{cl}$4$&for $\{h,a,e,f\}$\\
$+(n-1)^2$&for $C^\alpha{}_\beta$\\
$-1$&for the whole matrix being trace-free\\ \hline\hline
\rule{0pt}{12pt}\makebox[0pt][l]{$n^2-2n+4$}
\end{tabular}$$
and so the moduli space ${\mathrm{SL}}(n+2,{\mathbb{R}})/S$ of distinguished
curves of type~(c) in the flat model has dimension
$$\big((n+2)^2-1\big)-\big(n^2-2n+4\big)=6n-1,$$
in agreement with~(\ref{real_dimension}).

In this case, the Lie algebra ${\mathfrak{g}}={\mathfrak{sl}}(n+2)$ is 
$|2|$-graded. The case $n=3$ is sufficiently general to see what is happening to 
${\mathfrak{g}}=\oooo{1}{0}{0}{1}$:
$$\addtolength{\arraycolsep}{-3pt}\begin{array}{ccccccccc}
\xoox{1}{0}{0}{1}
&\oplus&
\!\begin{array}{c}\xoox{1}{0}{1}{-1}\\ \oplus\\
\xoox{-1}{1}{0}{1}\end{array}\!
&\oplus&
\!\begin{array}{c}\xoox{0}{0}{0}{0}\\ \oplus\\
\xoox{-1}{1}{1}{-1}\\ \oplus\\
\xoox{0}{0}{0}{0}\end{array}\!
&\oplus&
\!\begin{array}{c}\xoox{-2}{1}{0}{0}\\ \oplus\\
\xoox{0}{0}{1}{-2}\end{array}\!
&\oplus&\xoox{-1}{0}{0}{-1}\\
\|&&\|&&\|&&\|&&\|\\
{\mathfrak{g}}_{-2}&&{\mathfrak{g}}_{-1}&&{\mathfrak{g}}_0&&{\mathfrak{g}}_1
&&{\mathfrak{g}}_2\\[-10pt]
&&&&&&\makebox[0pt]{$
\underbrace{\hspace{200pt}}_{\mbox{\normalsize${\mathfrak{p}}.$}}$}
\end{array}$$
This confirms immediately that ${\mathfrak{a}}_\infty={\mathfrak{a}}_3$ and
a step-by-step calculation locates different pieces of the 
full symmetry algebra, starting with  
$$\left(\!\!\begin{array}{c}U^\alpha\\
V_\beta\end{array}\!\!\right)\in{\mathfrak{g}}_{-1}
=\begin{array}{c}\xoox{1}{0}{1}{-1}\\ \oplus\\
\xoox{-1}{1}{0}{1}\end{array}\quad\mbox{such that }U^\alpha V_\alpha\not=0.$$
The symmetry algebra ${\mathfrak{s}}$ comprises 
\begin{equation}\label{pieces_of_s}\begin{array}{ccccccccc}
{\mathfrak{g}}_{-2}&\oplus&{\mathfrak{g}}_{-1}&\oplus&{\mathfrak{g}}_0
&\oplus&{\mathfrak{g}}_1&\oplus&{\mathfrak{g}}_2\\
\inn&&\inn&&\inn&&\inn&&\inn\\
0&&f\left(\!\!\begin{array}{c}U^\alpha\\
V_\beta\end{array}\!\!\right)&&
\left(\!\!\begin{array}{c}a\\
C^\alpha{}_\beta\\
e\end{array}\!\!\right)
&&h\left(\!\!\begin{array}{c}
V_\beta\\
U^\alpha\end{array}\!\!\right)
&&0\makebox[0pt][l]{,}
\end{array}\end{equation}
where 
$$\textstyle a+C^\alpha{}_\alpha+e=0,\quad 
C^\alpha{}_\beta U^\beta=\frac{a+e}2U^\alpha,\quad
V_\alpha C^\alpha{}_\beta=\frac{a+e}2V_\beta.$$
\subsection{Contact Legendrean tractors}
To proceed, we need formul{\ae} for contact Legendrean tractors,
the standard tractor bundle being modelled on the standard representation of 
${\mathrm{SL}}(n+2,{\mathbb{R}})$, namely
$$\begin{array}{ccccccc}
\oooo{0}{0}{0}{1}
&=&\xoox{0}{0}{0}{1}&+&\xoox{0}{0}{1}{-1}&+&\xoox{-1}{0}{0}{0}\\
&&\|&&\|&&\|\\
&&\Wedge^0(0,1)&+&E(-1,0)&+&\Wedge^0(-1,0)\end{array},$$
where 
$$H=\begin{array}{c}E\\[-2pt] \oplus\\ F\end{array}
=\begin{array}{c}\xoox{1}{0}{1}{-1}\\ \oplus\\
\xoox{-1}{1}{0}{1}\end{array}$$
whose sections and tractor connection, in directions from $H$, we may write, in
a chosen exact scale following the conventions of~\cite[\S 5.2.15]{parabook},
as
\begin{equation}\label{tractor_connection_along_E}
\nabla_\alpha\left[\begin{array}c
\sigma \\ \mu^\beta \\ \rho
\end{array}\right]=
\left[\begin{array}c
\nabla_\alpha \sigma \\ 
\nabla_\alpha\mu^\beta+\delta_\alpha{}^\beta\rho+\Rho_\alpha{}^\beta\sigma\\
\nabla_\alpha\rho+A_{\alpha\beta}\mu^\beta+T_\alpha\sigma 
\end{array}\right]\end{equation}
and
\begin{equation}\label{tractor_connection_along_F}
\nabla^{\alpha}\left[\begin{array}c
\sigma \\ \mu^\beta \\ \rho
\end{array}\right]=
\left[\begin{array}c
\nabla^{\alpha}\sigma+\mu^\alpha\\ 
\nabla^{\alpha}\mu^\beta+A^{\alpha\beta}\sigma\\
\nabla^{\alpha}\rho+\Rho_\beta{}^\alpha\mu^\beta+T^\alpha\sigma 
\end{array}\right],\end{equation}
where
\begin{equation}\label{curvatures}
\begin{array}{rcl}\Rho_\alpha{}^\beta&\in
&\Gamma(\,\xoox{-2}{1}{0}{0}\otimes\xoox{0}{0}{1}{-2}\,)\\[4pt]
A_{\alpha\beta}&\in
&\Gamma(\,\xoox{-4}{2}{0}{0}\,)=\Gamma(\bigodot^2\xoox{-2}{1}{0}{0}\,)\\[4pt]
A^{\alpha\beta}&\in
&\Gamma(\,\xoox{0}{0}{2}{-4}\,)=\Gamma(\bigodot^2\xoox{0}{0}{1}{-2}\,),
\end{array}\end{equation}
are particular parts of the curvature whilst
$$T_\alpha\in\Gamma(\,\xoox{-3}{1}{0}{-1}\,)\quad\mbox{and}\quad
T^\alpha\in\Gamma(\,\xoox{-1}{0}{1}{-3}\,)$$
are given by 
\begin{equation}\label{Ts}\textstyle 
T_\alpha=\frac1{n+2}
\big(\nabla^\beta A_{\alpha\beta}-\nabla_\alpha\Rho_\beta{}^\beta\big)
\enskip\mbox{and}\enskip
T^\alpha=-\frac1{n+2}
\big(\nabla_\beta A^{\alpha\beta}-\nabla^\alpha\Rho_\beta{}^\beta\big).
\end{equation}
We shall also need the dual connection on cotractors:
\begin{equation}\label{cotractor_connection_along_E}
\nabla_\alpha
\left[\begin{array}c
\tau \\ \nu_{\beta} \\ \omega
\end{array}\right]=
\left[\begin{array}c
\nabla_{\alpha}\tau-\nu_{\alpha}\\ 
\nabla_{\alpha}\nu_{\beta}-A_{\alpha\beta}\tau\\
\nabla_{\alpha}\omega-\Rho_{\alpha}{}^{\beta}\nu_{\beta}-T_{\alpha}\tau
\end{array}\right]
\end{equation}
and
\begin{equation}\label{cotractor_connection_along_F}
\nabla^\alpha
\left[\begin{array}c
\tau \\ \nu_{\beta} \\ \omega
\end{array}\right]=
\left[\begin{array}c
\nabla^\alpha\tau\\
\nabla^\alpha\nu_{\beta}-\delta_\beta{}^\alpha\omega
-\Rho_\beta{}^\alpha\tau \\
\nabla^\alpha\omega-A^{\alpha\beta}\nu_{\beta}
-T^\alpha\tau 
\end{array}\right].
\end{equation}

\subsection{Adjoint tractors as endomorphisms}
Following our procedure in the conformal setting, we should view the adjoint
representation $\oooo{1}{0}{0}{1}$ of ${\mathrm{SL}}(n+2,{\mathbb{R}})$, here
written in case $n=3$, as the trace-free endomorphisms of the standard 
representation $\oooo{0}{0}{0}{1}$. Then, an adjoint tractor of the 
form (\ref{block_matrix}) acts by 
$$\left[\begin{array}c
\sigma \\ \mu^\beta \\ \rho
\end{array}\right]\longmapsto
\left[\begin{array}{ccc}e&Y_\beta&d\\
W^\alpha&C^\alpha{}_\beta&X^\alpha\\
b&Z_\beta&a\end{array}\right]\!\!\left[\begin{array}c
\sigma \\ \mu^\beta \\ \rho
\end{array}\right],$$
which, since
$$\left[\begin{array}{ccc}e&Y_\beta&d\\
W^\alpha&C^\alpha{}_\beta&X^\alpha\\
b&Z_\beta&a\end{array}\right]=\left[\begin{array}{ccc}0&0&1\\
0&{\mathrm{Id}}&0\\
1&0&0\end{array}\right]\!\!
\left[\begin{array}{ccc}a&Z_\beta&b\\
X^\alpha&C^\alpha{}_\beta&W^\alpha\\
d&Y_\beta&e\end{array}\right]\!\!
\left[\begin{array}{ccc}0&0&1\\
0&{\mathrm{Id}}&0\\
1&0&0\end{array}\right]$$
is consistent with Lie bracket being realised as matrix commutator. {From} this
viewpoint, the symmetry algebra~${\mathfrak{s}}$, as
in~(\ref{full_symmetry_algebra}), translates into a preferred class of
endomorphisms $\Phi$ of the standard tractor bundle. We see, for example, that
\begin{equation} \label{symmetry_on_tractors}
\Phi\!\left[\!\begin{array}c 0\\ 0\\ 1\end{array}\!\right]
\!=\!\left[\!\!\begin{array}c 0\\
f U^\alpha\\ 
a\end{array}\!\!\right]\quad
\Phi\!\left[\!\begin{array}c 0\\ U^\alpha\\ 0\end{array}\!\right]
\!=\!\left[\!\!\begin{array}c f\\
\frac{a + e}2U^\alpha\\ 
h\end{array}\!\!\right]\quad
\Phi\!\left[\!\begin{array}c 1\\ 0\\ 0\end{array}\!\right]
\!=\!\left[\!\!\begin{array}c e\\
hU^\alpha\\ 
0\end{array}\!\!\right]\!.
\end{equation}
To characterise~${\mathfrak{s}}$, we also need to encode that
$V_\alpha C^\alpha{}_\beta=\frac{a+e}2 V_\beta$ and for this it is convenient 
to use cotractors
$$\oooo{1}{0}{0}{0}=\xoox{1}{0}{0}{0}+\xoox{-1}{1}{0}{0}+\xoox{0}{0}{0}{-1},$$
the remaining condition required to characterise $\Phi$ as an endomorphism of 
cotractors being that 
\begin{equation}\label{remaining_condition}
\Phi\!\left[\!\begin{array}{c}0\\ V_\alpha\\ 0\end{array}\!\right]
=\left[\!\!\begin{array}{c}f\\ \frac{a+e}2 V_\alpha\\ h
\end{array}\!\!\right]\end{equation}
(in addition to
$$\Phi\!\left[\!\begin{array}{c}0\\ 0\\ 1\end{array}\!\right]
=\left[\!\!\begin{array}{c}0\\ fV_\alpha\\ e
\end{array}\!\!\right]\quad\mbox{and}\quad
\Phi\!\left[\!\begin{array}{c}1\\ 0\\ 0\end{array}\!\right]
=\left[\!\!\begin{array}{c}a\\ hV_\alpha\\ 0
\end{array}\!\!\right],$$
which can act as a check on consistency).

\subsection{The distinguished curves} We are in a position to determine, in the
contact Legendrean setting, whether a geodesic tangent to the contact
distribution and of `type~(c)' for an exact Weyl connection is distinguished as
an unparameterised curve. To be of `type~(c)' is precisely that its tangent
vector be of the form $(U^\alpha,V_\beta)$ with $U^\alpha V_\alpha\not=0$ and
we may suppose without loss of generality that $U^\alpha V_\alpha\equiv 1$
along~$\gamma$. To invoke Theorem~\ref{distinguished} we need a formula for the
tractor connection $\partial=U^\alpha\nabla_\alpha+V_\alpha\nabla^\alpha$ 
along~$\gamma$. According to (\ref{tractor_connection_along_E}) and 
(\ref{tractor_connection_along_F}) it is
$$\partial\left[\begin{array}c
\sigma \\ \mu^\beta \\ \rho
\end{array}\right]=
\left[\begin{array}c
\partial\sigma+V_\alpha\mu^\alpha\\ 
\partial\mu^\beta+U^\beta\rho+U^\alpha\Rho_\alpha{}^\beta\sigma
+V_\alpha A^{\alpha\beta}\sigma\\
\partial\rho+U^\alpha A_{\alpha\beta}\mu^\beta
+V_\alpha\Rho_\beta{}^\alpha\mu^\beta
+U^\alpha T_\alpha\sigma+V_\alpha T^\alpha\sigma
\end{array}\right].$$
As in the conformal case, we employ the Leibniz rule to compute 
$\partial\Phi$. Firstly, from (\ref{symmetry_on_tractors}) we find that
$$(\partial\Phi)\!\left[\!\begin{array}c 0\\ 0\\ 1\end{array}\!\right]
\!=\!\partial\!\left[\!\!\begin{array}c 0\\
f U^\beta\\ 
a\end{array}\!\!\right]
-\Phi\!\left[\!\!\begin{array}c 0\\
U^\beta\\ 
0\end{array}\!\!\right]\!=\!\left[\!\!\begin{array}c 0\\
(\partial f+\frac{a-e}2) U^\beta\\ 
\partial a+f\Lambda-h\end{array}\!\!\right],$$
where $\Lambda\equiv A_{\alpha\beta}U^\alpha U^\beta+\Rho_\alpha{}^\beta
U^\alpha V_\beta$. This is the same as the first condition from
(\ref{symmetry_on_tractors}) save for the replacements
$$\textstyle f\mapsto\tilde f\equiv\partial f+\frac{a-e}2\quad\mbox{and}\quad 
a\mapsto\tilde a\equiv\partial a-h+f\Lambda.$$
Next, we compute
$$\begin{array}{l}
(\partial\Phi)\!\left[\!\begin{array}c 0\\ U^\beta\\ 0\end{array}\!\right]
\!=\!\partial\!\left[\!\!\begin{array}c f\\
\frac{a+e}2 U^\beta\\ 
h\end{array}\!\!\right]
-\Phi\!\left[\!\!\begin{array}c 1\\
0\\ 
\Lambda
\end{array}\!\!\right]\\[16pt]
{}=\!\left[\!\!\begin{array}c\partial f+\frac{a+e}2\\
(\partial(\frac{a+e}2))U^\beta+hU^\beta+fU^\alpha\Rho_\alpha{}^\beta
+fV_\alpha A^{\alpha\beta}\\ 
\partial h+\frac{a+e}2\Lambda
+f(U^\alpha T_\alpha+V_\alpha T^\alpha)\end{array}\!\!\right]
-\left[\!\!\begin{array}c e\\
(h+f\Lambda)U^\beta\\ 
a\Lambda\end{array}\!\!\right]\\[16pt]
{}=\!\left[\!\!\begin{array}c\partial f+\frac{a-e}2\\
(\partial(\frac{a+e}2)-f\Lambda)U^\beta+fU^\alpha\Rho_\alpha{}^\beta
+fV_\alpha A^{\alpha\beta}\\ 
\partial h+\frac{e-a}2\Lambda
+f(U^\alpha T_\alpha+V_\alpha T^\alpha)\end{array}\!\!\right],\end{array}$$
which has the form
$$\left[\!\!\begin{array}c \tilde f\\
\frac{\tilde a+\tilde e}2 U^\beta\\ 
\tilde h\end{array}\!\!\right]$$
if and only if, in addition to $\textstyle\tilde f=\partial f+\frac{a-e}2$
as we already know,
$$\begin{array}{l}
U^\alpha\Rho_\alpha{}^\beta+V_\alpha A^{\alpha\beta}=\Xi U^\beta\\[4pt]
\tilde h=\partial h+\frac{e-a}2\Lambda+fK,\\[4pt]
\tilde e=\partial e+h+f(2\Xi-3\Lambda)\end{array}
\mbox{where $K\equiv U^\alpha T_\alpha+V_\alpha T^\alpha$}$$
for some smooth function $\Xi$ (and $\Lambda$ as above).
Now, we compute
$$\begin{array}{l}
(\partial\Phi)\!\left[\!\begin{array}c 1\\ 0\\ 0\end{array}\!\right]
\!=\!\partial\!\left[\!\!\begin{array}c e\\
hU^\beta\\
0\end{array}\!\!\right]
-\Phi\!\left[\!\!\begin{array}c 0\\
\Xi U^\beta\\ 
K
\end{array}\!\!\right]\\[16pt]
{}=\!\left[\!\!\begin{array}c\partial e+h\\
(\partial h)U^\beta+e\Xi U^\beta\\ 
h\Lambda+eK\end{array}\!\!\right]
-\left[\!\!\begin{array}c f\Xi\\
\frac{a+e}2 \Xi U^\beta+fK U^\beta\\ 
h\Xi+aK\end{array}\!\!\right]\\[16pt]
{}=\!\left[\!\!\begin{array}c\partial e+h-f\Xi\\
(\partial h+\frac{e-a}2\Xi-fK)U^\beta\\ 
(e-a)K-h(\Lambda-\Xi)\end{array}\!\!\right],\end{array}$$
which has the form
$$\left[\!\!\begin{array}c \tilde e\\
\tilde h U^\beta\\ 
0\end{array}\!\!\right]$$
if and only if
$$\tilde e=
\partial e+h-f\Xi\qquad\mbox{so}\enskip \Xi=\Lambda$$
and now 
$$K=0$$
and
$\textstyle\tilde h=\partial h+\frac{e-a}2\Xi=\partial h+\frac{e-a}2\Lambda$
is confirmed.
In summary, we have
$$\textstyle\tilde f=\partial f+\frac{a-e}2\qquad
\tilde h=\partial h+\frac{e-a}2\Lambda$$
$$\tilde a=\partial a-h+f\Lambda\qquad
\tilde e=\partial e+h-f\Lambda$$
for some smooth function $\Lambda$
and the following non-trivial conditions relating $(U^\alpha,V_\beta)$ and the
curvature of our chosen scale:
$$K\equiv U^\alpha T_\alpha+V_\alpha T^\alpha=0$$
and
$$U^\alpha\Rho_\alpha{}^\beta+V_\alpha A^{\alpha\beta}=\Lambda U^\beta,
\enskip\mbox{where }\Lambda\equiv 
A_{\alpha\beta}U^\alpha U^\beta+\Rho_\alpha{}^\beta U^\alpha V_\beta.$$

Finally, we need to determine the consequences of (\ref{remaining_condition}),
as the final restriction on~$\Phi$. For this we need a formula for the
cotractor connection $\partial=U^\alpha\nabla_a+V_\alpha\nabla^\alpha$
along~$\gamma$. According to (\ref{cotractor_connection_along_E}) and
(\ref{cotractor_connection_along_F}) it is
$$\partial
\left[\begin{array}c
\tau \\ \nu_{\beta} \\ \omega
\end{array}\right]=
\left[\begin{array}c
\partial\tau-U^\alpha\nu_{\alpha}\\ 
\partial\nu_{\beta}-U^\alpha A_{\alpha\beta}\tau-V_\beta\omega
-V_\alpha \Rho_\beta{}^\alpha\tau\\
\partial\omega-U^\alpha\Rho_{\alpha}{}^{\beta}\nu_{\beta}
-V_\alpha A^{\alpha\beta}\nu_{\beta}
-K\tau 
\end{array}\right],$$
with $K$ as above. Working through the consequences of 
(\ref{remaining_condition}) gives just one 
more consequence, namely
$$V_\alpha\Rho_\beta{}^\alpha+U^\alpha A_{\alpha\beta}=\Lambda V_\beta.$$
We have proved the following: 

\begin{thm}\label{contact_Legendrean_theorem}
Suppose that $M$, with contact distribution $H=E\oplus F$, is a contact
Legendrean manifold and that $(\nabla_\alpha,\nabla^\alpha)$ is an exact Weyl
connection in the contact directions. Suppose $\gamma\hookrightarrow M$ is a
geodesic for this connection, everywhere tangent to $H$ with tangent vector
$(U^\alpha,V_\alpha)$ such that $U^\alpha V_\alpha=1$. Then $\gamma$ is
distinguished as an unparameterised curve if and only if the following
constraints on curvature
\begin{equation}\label{constraints}
U^\alpha T_\alpha+V_\alpha T^\alpha=0\quad\mbox{and}\quad
\left[\!\begin{array}{cc}\Rho_\alpha{}^\beta&A^{\alpha\beta}\\
A_{\alpha\beta}&\Rho_\beta{}^\alpha\end{array}\!\right]
\left[\!\begin{array}{c}U^\alpha\\ V_\alpha\end{array}\!\right]
=\Lambda\left[\!\begin{array}{c}U^\beta\\ V_\beta\end{array}\!\right]
\end{equation}
are satisfied along~$\gamma$ for some scalar function $\Lambda$.
\end{thm}
\begin{cor} All of the unparameterised `type~(c)' contact geodesics of an exact
Weyl connection on a contact Legendrean manifold are distinguished if and only
if
\begin{equation}\label{cf_einstein}T_\alpha=0,\quad 
T^\alpha=0,\quad 
A_{\alpha\beta}=0,\quad
A^{\alpha\beta}=0,\quad
\Rho_\alpha{}^\beta=\lambda\delta_\alpha{}^\beta,\end{equation}
for some smooth function~$\lambda$.
\end{cor}
\begin{proof} It is elementary algebra to verify that the constraints
(\ref{constraints}) are satisfied for all $(U^\alpha,V_\alpha)$ such that 
$U^\alpha V_\alpha=1$ if and only if the equations (\ref{cf_einstein}) hold.
\end{proof}
Recall that the geometry that we are dealing with here is contact Legendrean
equipped with an exact Weyl connection, equivalently a nowhere vanishing 
{\em scale\/} $\sigma\in\Gamma(\xoox{1}{0}{0}{1})$, as
in~\cite[\S5.2.14]{parabook} where $\theta=1/\sigma\in
\Gamma(\,\xoox{-1}{0}{0}{-1}\,)\hookrightarrow\Gamma(\Wedge^1)$ is seen as a
choice of {\em contact form}. The equations (\ref{cf_einstein}) should be
regarded as the analogue of the Einstein equations in Riemannian geometry. The
following proposition supports this analogy. 
\begin{prop} If
\eqref{cf_einstein} hold, then $\lambda$ is constant.
\end{prop}
\begin{proof} Notice that, although $\lambda=\frac1n\Rho_\alpha{}^\alpha$ is,
in the first instance, a section of the line bundle $\,\xoox{-1}{0}{0}{-1}$, we
are working in the presence of a scale $\sigma\in\Gamma(\xoox{1}{0}{0}{1})$,
which trivialises this bundle. The formul{\ae} (\ref{Ts}) for $T_\alpha$ and
$T^\alpha$ now show that $\lambda$ is constant.
\end{proof}

\subsection{Examples}\label{contact_legendrean_examples} Given the strength of
the constraints (\ref{cf_einstein}), one might be concerned that they are only
satisfied for the flat model.  After all, the Einstein equations in three
dimensions imply constant curvature and one needs to look in four dimensions to
find non-trivial solutions.  Indeed, looking ahead to the corresponding and
more familiar equations (\ref{cf_einstein_CR}) in CR geometry, it turns out
that there are no non-trivial solutions to (\ref{cf_einstein}) in three
dimensions.  Hence, we should look in five dimensions and good places to look
are homogeneous structures.  Fortunately, the homogeneous contact Legendrean
structures (with isotropy) have been recently classified by Doubrov, Medvedev,
and The~\cite{DMT}.  Not only that, but their examples come equipped with a
natural choice of contact form.  We follow their notation in asserting that the
following models (some of which depend on one or two parameters)
\begin{center}
N.8,\quad N.7-1,\quad N.7-2,\quad N.6-1,\quad N.6-2
\end{center}
provide non-trivial solutions of (\ref{cf_einstein}) with~$\lambda=0$. More 
interesting is a particular model of the form D.7. Specifically, we may take 
the contact form $\sigma\equiv du-p\,dx-q\,dy$ with 
\begin{equation}\label{lambda_is_one}
E=\{\sigma,dx,dy\}^\perp\quad\mbox{and}\quad
F=\{\sigma,dp-p^2dx,dq-q^2dy\}^\perp\end{equation} 
in local co\"ordinates $(x,y,u,p,q)$. This produces a non-trivial solution to
(\ref{cf_einstein}) with $\lambda\equiv -1/3$. We remark that the computations
here are carried out using {\sc Maple}, which is able to deal with all models
from \cite{DMT} save for D.6-3\@. Even choosing an explicit parameter, this
model remains intractable (with our current program/implementation). The 
computations themselves are unilluminating and hence omitted.

\subsection{A first BGG operator}
By analogy with the conformal case, one might expect that the existence 
of a scale for which the equations (\ref{cf_einstein}) hold, is governed by a 
first BGG operator. Acting on scales, there is just one such operator:
$$\xoox{1}{0}{0}{1}\begin{array}{c}\nearrow\\ \searrow\end{array}\!\!
\begin{array}{c}\xoox{-3}{2}{0}{0}\\ \oplus\\ \xoox{0}{0}{2}{-3}\end{array}$$
and written in any chosen scale, it is given by
$$\sigma\mapsto
\left[\!\begin{array}{l}
\nabla_\alpha\nabla_\beta\sigma-A_{\alpha\beta}\sigma\\[4pt]
\nabla^\alpha\nabla^\beta\sigma-A^{\alpha\beta}\sigma
\end{array}\!\right]$$
where $\nabla_\alpha$ and $\nabla^\alpha$ are parts of the Weyl connection for
that scale and $A_{\alpha\beta}$ and $A^{\alpha\beta}$ are parts of the
corresponding curvature~(\ref{curvatures}). Evidently, if $\sigma\not=0$ is
in the kernel of this operator and we choose to view it in the scale defined 
by $\sigma$, then it becomes
$$1\mapsto
\left[\!\begin{array}{l}
\nabla_\alpha\nabla_\beta 1-A_{\alpha\beta}1\\[4pt]
\nabla^\alpha\nabla^\beta 1-A^{\alpha\beta}1
\end{array}\!\right]=\left[\!\begin{array}{l}
-A_{\alpha\beta}\\[4pt]
-A^{\alpha\beta}
\end{array}\!\right]$$
and we conclude that $A_{\alpha\beta}=0$ and $A^{\alpha\beta}=0$ in the 
scale~$\sigma$. These are part of the equations (\ref{cf_einstein}). It is 
unclear whether the remaining equations from (\ref{cf_einstein}) are captured 
by any BGG operator, although the equation 
$\Rho_\alpha{}^\beta=\lambda\delta_\alpha{}^\beta$ seems to be lurking in the 
formul{\ae} (\ref{tractor_connection_along_E}) and 
(\ref{tractor_connection_along_F}) for standard tractors.

\section{CR geometry}\label{CR_geometry}
Computing the symmetry algebra in the CR setting according to
Lemma~\ref{dkr_lemma} entails exactly the same arithmetic. Specifically, if we
take
$${\mathfrak{su}}(n+1,1)=\left\{\left[\begin{array}{ccc}
\lambda&-\bar{r}^t&iq\\
s&C&r\\
ip&-\bar{s}^t&-\bar\lambda\end{array}\right]\;\rule[-12pt]{.7pt}{30pt}\;
\begin{array}{l}C\mbox{ is skew Hermitian}\\
{\mbox{trace}}(C)+\lambda-\bar\lambda=0\end{array}\right\}$$
then we find that
$${\mathfrak{s}}=\left\{\!\left[\begin{array}{ccc}
\!x+i\theta&-h\bar{U}^t&0\\
fU&M+i\theta{\mathrm{Id}}&hU\\
0&-f\bar{U}^t&-x+i\theta\!\end{array}\right]\;\rule[-12pt]{.7pt}{30pt}\;
\begin{array}{l}M\mbox{ is skew Hermitian}\\
MU=0\\
{\mbox{trace}}(M)+(n+2)i\theta=0
\end{array}\!\right\},$$
which one recognises as the Lie algebra of~(\ref{this_is_S}) when $U$ is a
standard basis vector.

Similarly, the investigation of distinguished curves in CR geometry follows
exactly the contact Legendrean case. Indeed, one can pursue contact Legendrean
geometry in the holomorphic setting and then CR geometry is simply an
alternative real form. Thus, one arrives at the following result (with 
conventions from \cite{parabook,GG}).

\begin{thm}\label{CR_theorem}
Suppose $(M,H,J)$ is a CR manifold of hypersurface type and that $\theta$ is a
choice of contact form. Let $(\nabla_\alpha,\nabla_{\bar\alpha})$ denote the
exact Weyl connection in the contact directions that is defined by~$\theta$.
Suppose $\gamma\hookrightarrow M$ is a geodesic for this connection, everywhere
tangent to $H$ with tangent vector $(U^\alpha,V^{\bar\alpha})$ such that
$h_{\alpha\bar\beta}U^\alpha V^{\bar\beta}=1$, where $h_{\alpha\bar\beta}$ is
the Levi form, viewed in the scale defined by~$\theta$. Then $\gamma$ is
distinguished as an unparameterised curve if and only if the following
constraints on curvature
$$U^\alpha T_\alpha+V^{\bar\alpha}T_{\bar\alpha}=0\quad\mbox{and}\quad
\left[\!\begin{array}{cc}\Rho_\alpha{}^\beta&A_{\bar\alpha}{}^\beta\\
A_\alpha{}^{\bar\beta}&\Rho^{\bar\beta}{}_{\bar\alpha}\end{array}\!\right]
\left[\!\begin{array}{c}U^\alpha\\ V^{\bar\alpha}\end{array}\!\right]
=\Lambda\left[\!\begin{array}{c}U^\beta\\ V^{\bar\beta}\end{array}\!\right]$$
are satisfied along~$\gamma$ for some scalar function $\Lambda$ (where indices 
are raised and lowered using the Levi form~$h_{\alpha\bar\beta}$).
\end{thm}
\begin{cor} All of the unparameterised type~(c) contact geodesics of an exact
Weyl connection on a CR manifold of hypersurface type are distinguished if and
only if
$$T_\alpha=0,\quad 
T_{\bar\alpha}=0,\quad 
A_{\alpha\beta}=0,\quad
A_{\bar\alpha\bar\beta}=0,\quad
\Rho_{\alpha\bar\beta}=\lambda h_{\alpha\bar\beta},$$
for some smooth function~$\lambda$, which is then necessarily constant.
\end{cor}
\noindent{\bf Remark}\enskip Of course, in the CR setting these curvature are 
`real' in the sense that
$$T_{\bar\alpha}=\overline{T_\alpha},\quad 
A_{\bar\alpha\bar\beta}=\overline{A_{\alpha\beta}},\quad 
\Rho_{\alpha\bar\beta}=\overline{\Rho_{\beta\bar\alpha}},$$
so the constraints in this corollary amount to
\begin{equation}\label{cf_einstein_CR}T_\alpha=0,\quad A_{\alpha\beta}=0,\quad
\Rho_{\alpha\bar\beta}=\lambda h_{\alpha\bar\beta},\end{equation} 
where $\lambda\in{\mathbb{R}}$.  The last of these equations was already
investigated by Lee~\cite{Lee}, who adopts the terminology `pseudo-Einstein'
for its solutions in line with the `pseudo-Hermitian' terminology of
Webster~\cite{W} for a CR structure equipped with a choice of contact form.
Lee~\cite{Lee} already remarks in his introduction that this condition is `very
different from its analogue in Riemannian geometry\ldots\ such a structure is
less rigid than an Einstein structure on a Riemannian manifold.'  The
additional equations in (\ref{cf_einstein_CR}) seem to restore the analogy 
with Einstein metrics.

We remark that the constraint $A_{\alpha\beta}=0$ has also been considered in
the literature.  The tensor $A_{\alpha\beta}$ is called the Webster torsion and
its vanishing picks out the {\em transversally symmetric\/} pseudo-Hermitian
geometries~\cite{Lee}, i.e.\ those for which the associated Reeb field is an
infinitesimal CR symmetry.  The pseudo-Hermitian geometries enjoying
(\ref{cf_einstein_CR}) are dubbed `TS-pseudo-Einstein' by
Leitner~\cite{Leitner} who constructs examples thereof on $S^1$ bundles over
K\"ahler-Einstein manifolds. In~\cite{CG}, \v{C}ap and Gover give a tractor 
formulation of the TS-pseudo-Einstein structures, which further strengthens the 
analogy with Einstein metrics in conformal geometry.

We should explain what happens when $n=1$, i.e.~when $M$ is
three-dimensional.  The equation $\Rho_{\alpha\bar\beta}=\lambda
g_{\alpha\bar\beta}$ is always satisfied and $\lambda$ is the Webster scalar
curvature.  The tensor $A_{\alpha\beta}$ is equivalent to a complex-valued
function (see e.g.~\cite{EN} for an exposition in three dimensions). These two 
invariants `solve the equivalence problem' in this context and, in particular, 
if $\lambda=1$ and $A_{\alpha\bar\beta}=0$, then $M$ is the sphere 
$S^3\subset{\mathbb{C}}^2$ 
equipped with its standard CR structure and contact form. 

\medskip\noindent{\bf Remark}\enskip Regarding examples, we may consider the
five-dimensional CR manifolds with many symmetries. Following~\cite{K}, example
N.7-2 from \S\ref{contact_legendrean_examples} has a Levi-definite real form
$$\{(z_1,z_2,z_3)\in{\mathbb{C}}^3\mid
{\mathrm{Im}}(z_3)=\log(1+|z_1|^2)+|z_2|^2\}$$
satisfying (\ref{cf_einstein_CR}) with~$\lambda=0$.
The more interesting example~(\ref{lambda_is_one}), with $\lambda=-1/3$, 
has a Levi-indefinite version,
which appears in \cite{L} as
$$\{(z_1,z_2,z_3)\in{\mathbb{C}}^3\mid
{\mathrm{Im}}(z_3)=\log(1+z_1\overline{z_2})+\log(1+\overline{z_1}z_2)\}$$
and in~\cite{DMT_CR} as the tube in ${\mathbb{C}}^3$ over the affine
homogeneous surface
$$\{(x,y,u)\in{\mathbb{R}}^3\mid u=\log(xy)\}.$$
As observed in~\cite{DMT_CR}, and pointed out to us by Dennis The, this example
also has a Levi-definite incarnation, namely as the tube over the affine
homogeneous surface
$$\{(x,y,u)\in{\mathbb{R}}^3\mid u=\log(x^2+y^2)\}.$$

\section*{Appendix: on the equation for conformal circles}
Of course, as we saw in the flat model, the reason that distinguished curves in
conformal geometry are called {\em conformal circles\/} is that they are
modelled on circles in ${\mathbb{R}}^n$. As mentioned
in~(\ref{conformal_Killing_fields}), the Lie algebra~${\mathfrak{g}}$, to be
used in the Doubrov-\v{Z}\'adn\'{\i}k~\cite{DZ} characterisation of conformal
circles, is the algebra of conformal Killing fields
$$\textstyle 
(X^b-F^b{}_cx^c+\lambda x^b-Y_cx^cx^b+\frac12x_cx^cY^b)
\,\partial/\partial x^b$$
for constant tensors $X^b,F^b{}_c,\lambda,Y_b$ with $F_{bc}$ skew. The circles 
through $0\in{\mathbb{R}}^n$ may be parameterised by a pair of vectors
\begin{equation}\label{pair}
U^a\mbox{ of unit length},\qquad C^a\mbox{ such that }U^aC_a=0.\end{equation}
Indeed, it is readily verified that 
\begin{equation}\label{circle}
{\mathbb{R}}\ni t\mapsto\frac{2}{1+t^2C^aC_a}(tU^b+t^2C^b)\end{equation}
is the circle with velocity $U^a$ and acceleration $C^a$ at the origin, where 
we are allowing straight lines as `circles' with $C^a=0$.
\begin{lemma}\label{symmetry_algebra_of_a_circle}
The symmetry algebra of the circle \eqref{circle} is specified 
by the linear constraints
\begin{equation}\label{symL}
X^b=fU^b\qquad F^b{}_cU^c=-fC^b\qquad Y^b=hU^b+\lambda C^b+F^b{}_cC^c
\end{equation}
where $f$, $\lambda$, and $h$ are arbitrary.  
\end{lemma}
\begin{proof} An elementary verification.\end{proof}
Note, in particular, that since the symmetry algebra of a circle is non-trivial
(of dimension $\big((n-1)(n-2)/2\big)+3$), circles are homogeneous. It is 
well-known that circles are preserved by conformal transformations. This is 
why they are suitable model curves in conformal geometry. We also take the 
opportunity to note that the moduli space of conformal circles in $S^n$ has 
dimension given by
$$\begin{array}{ccccc}\dim{\mathrm{SO}}(n+1,1)&-&
\dim{\mathrm{Sym}}\,\mbox{circle}\\
\|&&\|\\
(n+2)(n+1)/2&-&\big((n-1)(n-2)/2\big)+3&=&3(n-1),
\end{array}$$
as observed geometrically in the Introduction.

To unpack the Doubrov-\v{Z}\'adn\'{\i}k characterisation, we shall use the
conformal tractor bundle and its connection given, in the presence of a metric,
by~(\ref{conformal_tractor_connection}). We shall also need the general form of
a tractor endomorphism as in (\ref{endomorphisms}) and the tractor directional
derivative along $\gamma$ as in~(\ref{directional_derivative}).
We conclude that $\gamma\hookrightarrow M$ is a conformal circle if
and only if we can find a subbundle ${\mathcal{S}}$ of the endomorphisms of
${\mathcal{T}}$ along~$\gamma$, having the form (\ref{endomorphisms}) and such
that
\begin{itemize}
\item $X^a$ is tangent to $\gamma$,
\item fibrewise, the subbundle ${\mathcal{S}}$ has the form~(\ref{symL}),
\item ${\mathcal{S}}$ is preserved by the tractor connection $\partial\equiv 
U^a\nabla_a$ along~$\gamma$.
\end{itemize}
To proceed, it is useful to reformulate the conditions (\ref{symL}) as follows.
\begin{lemma}\label{REFORMULATE}
In order that an endomorphism \eqref{endomorphisms}
satisfy~\eqref{symL}, for some fields $U^a$ and $C^a$ satisfying~\eqref{pair},
it is necessary and sufficient that
\begin{equation}\label{NEC_AND_SUFF}\begin{array}{c}
\Phi\left[\begin{array}c0\\ 0\\ 1\end{array}\right]
=\left[\begin{array}c 0\\
-fU_b\\ 
\lambda\end{array}\right]\qquad
\Phi\left[\begin{array}c0\\ U_b\\ 0\end{array}\right]
=\left[\begin{array}c f\\
-fC_b\\ 
-h-fC^bC_b\end{array}\right]\\[-5pt] \\
\Phi\left[\begin{array}c 1\\ -C_b\\ 0\end{array}\right]
=\left[\begin{array}c -\lambda\\
hU_b+\lambda C_b\\ 
\lambda C^bC_b\end{array}\right].
\end{array}\end{equation}
\end{lemma}
\begin{proof}
If (\ref{symL}) are satisfied, then it is straightforward to verify that all
equations of (\ref{NEC_AND_SUFF}) hold. Conversely, notice that the first two
equations from (\ref{NEC_AND_SUFF}) determine $f$, $U^a$, $\lambda$, $C^a$, and
$h$. The first two constraints from (\ref{symL}) are manifest. If we now
substitute into (\ref{endomorphisms}), then we find that
$$\Phi\left[\begin{array}c 1\\ 0\\ 0\end{array}\right]
=\left[\begin{array}c -\lambda\\
Y_b\\ 
0\end{array}\right]\quad\mbox{and}\quad
\Phi\left[\begin{array}c0\\ C_b\\ 0\end{array}\right]
=\left[\begin{array}c 0\\
F_{bc}C^c\\ 
-Y_bC^b\end{array}\right].$$
{From} the second component of the last equation of (\ref{NEC_AND_SUFF}) we
conclude that $Y_b=hU_b+\lambda C_b+F_{bc}C^c$, which is the final 
constraint from (\ref{symL}).\end{proof}
In fact, three of the conditions in (\ref{NEC_AND_SUFF}) are automatic as 
follows.
\begin{lemma}\label{compact_reformulate}
In order that an endomorphism \eqref{endomorphisms}
satisfy~\eqref{symL}, for some fields $U^a$ and $C^a$ satisfying~\eqref{pair},
it is necessary and sufficient that
\begin{equation}\label{compact_nec_and_suff}
\Phi\!\left[\begin{array}c\!0\!\\ \!0\!\\ \!1\!\end{array}\right]
\!\!=\!\!\left[\begin{array}c 0\\
\!\!-fU_b\!\!\\ 
\lambda\end{array}\right]\enskip
\Phi\!\left[\begin{array}c0\\ \!U_b\!\\ 0\end{array}\right]
\!\!=\!\!\left[\begin{array}c f\\
\!\!-fC_b\!\!\\ 
*\end{array}\right]\enskip
\Phi\!\left[\begin{array}c 1\\ \!\!-C_b\!\!\\ 0\end{array}\right]
\!\!=\!\!\left[\begin{array}c *\\
\!\!hU_b\!+\!\lambda C_b\!\!\\ 
*\end{array}\right]\!\!.
\end{equation}
\end{lemma}
\begin{proof} In proving Lemma~\ref{REFORMULATE} we never used the 
starred components.
\end{proof}
We aim to interpret Theorem~\ref{distinguished} as a system of ordinary
differential equations on the velocity $U^a$ and acceleration $C^a$
of~$\gamma$. Since $X^a$ should be tangent to $\gamma$ we may write~$X^b=fU^b$.
This is the first constraint from (\ref{symL}) and, of course, our conventions
have been chosen for this to be the case. Our next observation justifies $C^a$
as our choice of notation for the acceleration of both $\gamma$ and the
corresponding model circle in~${\mathbb{R}}^n$.
\begin{lemma}\label{acceleration}
In order to have a subbundle ${\mathcal{S}}$ of ${\mathcal{A}}|_\gamma$,
constrained by~\eqref{symL} and being preserved by the tractor directional
derivative $\partial$ along~$\gamma$, it is necessary that the field $C^b$ in
\eqref{symL} be the acceleration $\partial U^b$ of~$\gamma$ (defined with
respect to our choice of metric~$g_{ab}$).
\end{lemma}
\begin{proof}
According to Lemma~\ref{REFORMULATE} the equations (\ref{symL}) are equivalent
to (\ref{NEC_AND_SUFF}) and, from~(\ref{conformal_tractor_connection}), the
tractor connection $\partial=U^a\nabla_a$ along $\gamma$ is given by
\begin{equation}\label{DIRECTIONAL_DERIVATIVE}
\partial\left[\begin{array}c\sigma\\ \mu_b\\ \rho\end{array}\right]= 
\left[\begin{array}c\partial\sigma-U^a\mu_a\\
\partial\mu_b+U_b\rho+U^a\Rho_{ab}\sigma\\ 
\partial\rho-U^a\Rho_a{}^b\mu_b\end{array}\right].\end{equation}
The Leibniz rule now calculates the effect of $\partial\Phi$ and, in 
particular, we find from (\ref{NEC_AND_SUFF}) that 
$$(\partial\Phi)\!\left[\begin{array}c0\\ 0\\ 1\end{array}\right]
\!=\!\partial\!\left[\begin{array}c 0\\
\!\!-fU_b\!\\ \lambda\end{array}\right]-
\Phi\partial\!
\left[\begin{array}c0\\ 0\\ 1\end{array}\right]
\!=\!\left[\begin{array}c0\\ 
\!\!-(\partial f-\lambda)U_b-f(\partial U_b-C_b)\!\\
\!\partial\lambda\!+\!h\!+\!f(C^aC_a\!+\!\Rho_{ab}U^aU^b)
\end{array}\right]\!\!.$$
Comparison with (\ref{NEC_AND_SUFF}) shows that $C_b=\partial U_b$, as 
required. 
\end{proof}

For a full interpretation of Theorem~\ref{distinguished} in conformal geometry
we are obliged to investigate what it means for the tractor directional
derivative $\partial$ to preserve all the equations from (\ref{NEC_AND_SUFF}).
Whilst certainly possible, some equations are automatic and, according to
Lemma~\ref{compact_reformulate}, it suffices to start with an endomorphism
$\Phi$ satisfying (\ref{NEC_AND_SUFF}) and investigate what it means for 
$\partial\Phi$ to satisfy~(\ref{compact_nec_and_suff}). 
This investigation was begun in the proof of Lemma~\ref{acceleration} above 
and next we should consider
$$(\partial\Phi)\left[\begin{array}c0\\ U_b\\ 0\end{array}\right]
=\partial\Big(\Phi\left[\begin{array}c0\\ U_b\\ 0\end{array}\right]\Big)
-\Phi\partial\left[\begin{array}c0\\ U_b\\ 0\end{array}\right]\!,$$
the right hand side of which may be computed from (\ref{NEC_AND_SUFF})
and~(\ref{DIRECTIONAL_DERIVATIVE}). Indeed, we find
$$\begin{array}{rcl}
\partial\Big(\Phi\left[\begin{array}c0\\ U_b\\ 0\end{array}\right]\Big)
&=&\partial\left[\begin{array}c f\\
-fC_b\\ 
-h-fC^bC_b\end{array}\right]\\[-5pt] \\
&=&\left[\begin{array}c\partial f\\ 
-\partial(fC_b)-hU_b-fC^aC_aU_b+fU^a\Rho_{ab}\\ 
-\partial(h+fC^aC_a)+f\Rho_{ab}U^aC^b\end{array}\right]
\end{array}$$
and
$$\Phi\partial\left[\begin{array}c0\\ U_b\\ 0\end{array}\right]
=\Phi\left[\begin{array}c-1\\
C_b\\ 
-\Rho_{ab}U^aU^b\end{array}\right]
=\left[\begin{array}c \lambda\\
(f\Rho_{ac}U^aU^c-h)U_b-\lambda C_b\\ 
-\lambda(C^bC_b+\Rho_{bc}U^bU^c)\end{array}\right]$$
so 
$$(\partial\Phi)\left[\begin{array}c0\\ U_b\\ 0\end{array}\right]
=\left[\begin{array}c\partial f-\lambda\\ 
-(\partial f-\lambda)C_b
-fE_b\\ *\end{array}\right],$$
where $E_b\equiv\partial C_b-\Rho_{bc}U^c+(C^aC_a+\Rho_{ac}U^aU^c)U_b$ and, 
for our present purposes, it does not matter what is the last entry. We 
conclude that ${\mathcal{S}}$ is preserved along $\gamma$ in accordance with 
Theorem~\ref{distinguished}, if and only if $E_b$ is identically zero, which 
is the conformal circles equation~(\ref{cc}). 

Finally, to complete our investigation of (\ref{compact_nec_and_suff}) we 
should compute
$$(\partial\Phi)\left[\begin{array}c1\\ -C_b\\ 0\end{array}\right]
=\partial\Big(\Phi\left[\begin{array}c1\\ -C_b\\ 0\end{array}\right]\Big)
-\Phi\partial\left[\begin{array}c1\\ -C_b\\ 0\end{array}\right]\!.$$
Well, from (\ref{NEC_AND_SUFF}), we find 
$$\begin{array}{rcl}
\partial\Big(\Phi\left[\begin{array}c1\\ -C_b\\ 0\end{array}\right]\Big)
&=&\partial\left[\begin{array}c -\lambda\\
hU_b+\lambda C_b\\ 
\lambda C^aC_a\end{array}\right]\\[-5pt] \\
&=&\left[\begin{array}c -\partial\lambda-h\\
\partial(hU_b+\lambda C_b)+\lambda C^aC_aU_b-\lambda\Rho_{ab}U^a\\ 
*\end{array}\right]
\end{array}$$
and
$$\Phi\partial\left[\begin{array}c1\\ -C_b\\ 0\end{array}\right]
=\Phi\left[\begin{array}c 0\\
-\partial C_b+U^a\Rho_{ab}\\ 
U^a\Rho_a{}^bC_b\end{array}\right]
=\Phi\left[\begin{array}c 0\\
(C^aC_a+\Rho_{ac}U^aU^c)U_b\\ 
\Rho_{ab}U^aC^b\end{array}\right],$$
where we have just used (\ref{cc}) to rewrite $\partial C_b-\Rho_{ab}U^a$. 
From here,
$$\Phi\partial\left[\begin{array}c1\\ -C_b\\ 0\end{array}\right]
=\left[\begin{array}c f(C^aC_a+\Rho_{ac}U^aU^c)\\
-f\Rho_{ac}U^aC^cU_b-f(C^aC_a+\Rho_{ac}U^aU^c)C_b\\ 
* \end{array}\right].$$
Also note that (\ref{cc}) allows us to write 
$$\partial(hU_b+\lambda C_b)+\lambda C^aC_aU_b-\lambda\Rho_{ab}U^a
=(\partial h-\lambda\Rho_{ac}U^aU^c)U_b+(\partial\lambda+h)C_b$$
and so, finally,
$$(\partial\Phi)\left[\begin{array}c1\\ -C_b\\ 0\end{array}\right]
=\left[\begin{array}c 
-\partial\lambda-h-f(C^aC_a+\Rho_{ac}U^aU^c)\\[2pt]
\begin{array}{l}(\partial h-\lambda\Rho_{ac}U^aU^c+f\Rho_{ac}U^aC^c)U_b\\[-1pt]
\qquad{}+(\partial\lambda+h+f(C^aC_a+\Rho_{ac}U^aU^c))C_b\end{array}\\[8pt]
*\end{array}\right],$$
in accordance with~(\ref{compact_nec_and_suff}). We have proved the following.
\begin{thm} Let $M$ be a Riemannian manifold. An unparameterised curve
$\gamma\hookrightarrow M$ is distinguished in the sense of
Theorem~\ref{distinguished}, modelled on the circles in ${\mathbb{R}}^n$, if 
and only if \eqref{cc} holds along~$\gamma$. 
\end{thm}

\raggedright\raggedbottom

\end{document}